\documentclass[oneside,english]{amsart}
\usepackage[T1]{fontenc}
\usepackage[latin9]{inputenc}
\usepackage{babel}
\usepackage{amstext}
\usepackage{amsthm}
\usepackage{amssymb}
\PassOptionsToPackage{normalem}{ulem}
\usepackage{ulem}
\usepackage[unicode=true,pdfusetitle,
 bookmarks=true,bookmarksnumbered=false,bookmarksopen=false,
 breaklinks=false,pdfborder={0 0 1},backref=section,colorlinks=false]
 {hyperref}

\makeatletter
\numberwithin{equation}{section}
\numberwithin{figure}{section}
\theoremstyle{plain}
\newtheorem{thm}{\protect\theoremname}[section]
\theoremstyle{plain}
\newtheorem{fact}[thm]{\protect\factname}
\theoremstyle{remark}
\newtheorem{rem}[thm]{\protect\remarkname}
\theoremstyle{definition}
\newtheorem{example}[thm]{\protect\examplename}
\theoremstyle{plain}
\newtheorem{prop}[thm]{\protect\propositionname}
\theoremstyle{definition}
\newtheorem{defn}[thm]{\protect\definitionname}
\theoremstyle{plain}
\newtheorem{cor}[thm]{\protect\corollaryname}
\theoremstyle{remark}
\newtheorem{claim}[thm]{\protect\claimname}
\theoremstyle{plain}
\newtheorem{lem}[thm]{\protect\lemmaname}
\theoremstyle{plain}
\newtheorem{conjecture}[thm]{\protect\conjecturename}
\theoremstyle{plain}
\newtheorem{question}[thm]{\protect\questionname}

\usepackage{amsfonts}
\usepackage{nicefrac}
\usepackage{amscd}
\usepackage{a4wide}
\usepackage{url}

\AtBeginDocument{
  
}

\makeatother

\providecommand{\claimname}{Claim}
\providecommand{\conjecturename}{Conjecture}
\providecommand{\corollaryname}{Corollary}
\providecommand{\definitionname}{Definition}
\providecommand{\examplename}{Example}
\providecommand{\factname}{Fact}
\providecommand{\lemmaname}{Lemma}
\providecommand{\propositionname}{Proposition}
\providecommand{\questionname}{Question}
\providecommand{\remarkname}{Remark}
\providecommand{\theoremname}{Theorem}

\begin{document}
\global\long\def\code#1{\ulcorner#1\urcorner}%
\global\long\def\p{\mathbf{p}}%
\global\long\def\q{\mathbf{q}}%
\global\long\def\C{\mathfrak{C}}%
\global\long\def\SS{\mathcal{P}}%
 
\global\long\def\pr{\operatorname{pr}}%
\global\long\def\image{\operatorname{im}}%
\global\long\def\otp{\operatorname{otp}}%
\global\long\def\dec{\operatorname{dec}}%
\global\long\def\suc{\operatorname{suc}}%
\global\long\def\pre{\operatorname{pre}}%
\global\long\def\qe{\operatorname{qf}}%
 
\global\long\def\ind{\operatorname{ind}}%
\global\long\def\Nind{\operatorname{Nind}}%
\global\long\def\lev{\operatorname{lev}}%
\global\long\def\Suc{\operatorname{Suc}}%
\global\long\def\HNind{\operatorname{HNind}}%
\global\long\def\minb{{\lim}}%
\global\long\def\concat{\frown}%
\global\long\def\cl{\operatorname{cl}}%
\global\long\def\tp{\operatorname{tp}}%
\global\long\def\id{\operatorname{id}}%
\global\long\def\cons{\left(\star\right)}%
\global\long\def\qf{\operatorname{qf}}%
\global\long\def\ai{\operatorname{ai}}%
\global\long\def\dtp{\operatorname{dtp}}%
\global\long\def\acl{\operatorname{acl}}%
\global\long\def\nb{\operatorname{nb}}%
\global\long\def\limb{{\lim}}%
\global\long\def\leftexp#1#2{{\vphantom{#2}}^{#1}{#2}}%
\global\long\def\intr{\operatorname{interval}}%
\global\long\def\atom{\emph{at}}%
\global\long\def\I{\mathfrak{I}}%
\global\long\def\uf{\operatorname{uf}}%
\global\long\def\ded{\operatorname{ded}}%
\global\long\def\Ded{\operatorname{Ded}}%
\global\long\def\Df{\operatorname{Df}}%
\global\long\def\Th{\operatorname{Th}}%
\global\long\def\eq{\operatorname{eq}}%
\global\long\def\Aut{\operatorname{Aut}}%
\global\long\def\ac{ac}%
\global\long\def\DfOne{\operatorname{df}_{\operatorname{iso}}}%
\global\long\def\modp#1{\pmod#1}%
\global\long\def\sequence#1#2{\left\langle #1\,\middle|\,#2\right\rangle }%
\global\long\def\set#1#2{\left\{  #1\,\middle|\,#2\right\}  }%
\global\long\def\Diag{\operatorname{Diag}}%
\global\long\def\Nn{\mathbb{N}}%
\global\long\def\mathrela#1{\mathrel{#1}}%
\global\long\def\twiddle{\mathord{\sim}}%
\global\long\def\mathordi#1{\mathord{#1}}%
\global\long\def\Qq{\mathbb{Q}}%
\global\long\def\dense{\operatorname{dense}}%
\global\long\def\Rr{\mathbb{R}}%
 
\global\long\def\cof{\operatorname{cof}}%
\global\long\def\tr{\operatorname{tr}}%
\global\long\def\treeexp#1#2{#1^{\left\langle #2\right\rangle _{\tr}}}%
\global\long\def\x{\times}%
\global\long\def\forces{\Vdash}%
\global\long\def\Vv{\mathbb{V}}%
\global\long\def\Uu{\mathbb{U}}%
\global\long\def\tauname{\dot{\tau}}%
\global\long\def\ScottPsi{\Psi}%
\global\long\def\cont{2^{\aleph_{0}}}%
\global\long\def\MA#1{{MA}_{#1}}%
\global\long\def\rank#1#2{R_{#1}\left(#2\right)}%
\global\long\def\cal#1{\mathcal{#1}}%

\def\Ind#1#2{#1\setbox0=\hbox{$#1x$}\kern\wd0\hbox to 0pt{\hss$#1\mid$\hss} \lower.9\ht0\hbox to 0pt{\hss$#1\smile$\hss}\kern\wd0} 
\def\Notind#1#2{#1\setbox0=\hbox{$#1x$}\kern\wd0\hbox to 0pt{\mathchardef \nn="3236\hss$#1\nn$\kern1.4\wd0\hss}\hbox to 0pt{\hss$#1\mid$\hss}\lower.9\ht0 \hbox to 0pt{\hss$#1\smile$\hss}\kern\wd0} 
\def\nind{\mathop{\mathpalette\Notind{}}} 

\global\long\def\ind{\mathop{\mathpalette\Ind{}}}%
\global\long\def\Age{\operatorname{Age}}%
\global\long\def\lex{\operatorname{lex}}%
\global\long\def\len{\operatorname{len}}%

\global\long\def\dom{\operatorname{Dom}}%
\global\long\def\res{\operatorname{res}}%
\global\long\def\alg{\operatorname{alg}}%
\global\long\def\dcl{\operatorname{dcl}}%
 
\global\long\def\nind{\mathop{\mathpalette\Notind{}}}%
\global\long\def\average#1#2#3{Av_{#3}\left(#1/#2\right)}%
\global\long\def\Ff{\mathfrak{F}}%
\global\long\def\mx#1{Mx_{#1}}%
\global\long\def\maps{\mathfrak{L}}%

\global\long\def\Esat{E_{\mbox{sat}}}%
\global\long\def\Ebnf{E_{\mbox{rep}}}%
\global\long\def\Ecom{E_{\mbox{com}}}%
\global\long\def\BtypesA{S_{\Bb}^{x}\left(A\right)}%
\global\long\def\DenseTrees{T_{dt}}%

\global\long\def\init{\trianglelefteq}%
\global\long\def\fini{\trianglerighteq}%
\global\long\def\Bb{\cal B}%
\global\long\def\Lim{\operatorname{Lim}}%
\global\long\def\Succ{\operatorname{Succ}}%

\global\long\def\SquareClass{\cal M}%
\global\long\def\leqstar{\leq_{*}}%
\global\long\def\average#1#2#3{Av_{#3}\left(#1/#2\right)}%
\global\long\def\cut#1{\mathfrak{#1}}%
\global\long\def\NTPT{\text{NTP}_{2}}%
\global\long\def\Zz{\mathbb{Z}}%
\global\long\def\TPT{\text{TP}_{2}}%
\global\long\def\supp{\operatorname{supp}}%

\global\long\def\OurSequence{\mathcal{I}}%
\global\long\def\SUR{SU}%
\global\long\def\ShiftGraph#1#2{Sh_{#2}\left(#1\right)}%
\global\long\def\ShiftGraphHalf#1{\cal G_{#1}^{\frac{1}{2}}}%
\global\long\def\SymShiftGraph#1#2{Sh_{#2}^{sym}\left(#1\right)}%

\title{Automorphism groups of finite topological rank}
\author{Itay Kaplan and Pierre Simon}
\thanks{The first author would like to thank the Israel Science foundation
for their support of this research (Grant no. 1533/14). The second
author was partially supported by NSF (grant DMS 1665491), and the
Sloan foundation.}
\begin{abstract}
We offer a criterion for showing that the automorphism group of an
ultrahomogeneous structure is topologically 2-generated and even has
a cyclically dense conjugacy class. We then show how finite topological
rank of the automorphism group of an $\omega$-categorical structure
can go down to reducts. Together, those results prove that a large
number of $\omega$-categorical structures that appear in the literature
have an automorphism group of finite topological rank. In fact, we
are not aware of any $\omega$-categorical structure to which they
do not apply (assuming the automorphism group has no compact quotients).
We end with a few questions and conjectures.
\end{abstract}

\subjclass[2010]{20B27, 03C15. }
\maketitle

\section{Introduction}

Many automorphism groups of Fra\"iss\'e structures are known to
admit a 2-generated dense subgroup. This is the case for example for
dense linear orders and the random graph \cite{mac_automorphisms,Darji:2010}.
It is however not true that all automorphism groups of say $\omega$-categorical
structures have even a finitely generated dense subgroup. For instance
a construction of Cherlin and Hrushovski yields an $\omega$-categorical
structure whose automorphism groups admits $(\mathbb{Z}/2\mathbb{Z})^{\omega}$
as a quotient, which implies that it cannot have a finitely generated
dense subgroup (see Remark \ref{rem:constraint for cdcc compact quotient}).
However, it seems that the existence of such a large compact quotient
is the only known obstruction. We speculate that this might indeed
be the case and ask: Let $G$ is the automorphism group of an $\omega$-categorical
structure; assume $G$ has no compact quotient, then does it have
a finitely generated dense subgroup? This paper is our attempt at
answering this question. We fall short of providing a definitive answer,
but we succeed in finding sufficient conditions for such a $G$ to
admit a finitely generated dense subgroup which seem to apply to all
known examples. It is even plausible that those conditions are actually
satisfied by all $\omega$-categorical structures with trivial $\acl^{\eq}\left(\emptyset\right)$,
see Conjecture \ref{conj:expansion CIR}.

We now describe our main results. We first define a notion of a \emph{canonical
independence relation}, or CIR. It is a ternary independence relation
$\ind$ which satisfies in particular stationarity over $\emptyset$
and transitivity on both sides. Importantly, we do not assume symmetry.
We show that if an ultrahomogeneous structure $M$ admits such a CIR,
then it has a 2-generated dense subgroup, and even a cyclically dense
conjugacy class (that is, for some $f,g\in G$, the set $\set{f^{-n}gf^{n}}{n\in\Zz}$
is dense). Many, but not all, classical $\omega$-categorical structures
have a CIR. Examples that do not include the dense circular order
(Corollary \ref{cor:Circular order has no CIR but degree <=00003D3})
and the dense infinitely-branching tree (Corollary \ref{cor:Dense trees don't have a groovy ind relation}).
However, an expansion of any of those structures obtained by naming
a point does have a CIR (Example \ref{exa:trees with points}, Remark
\ref{rem:expanding a circular order by a point has a CIR}). This
leads to our second main theorem which completely solves a relative
version of our initial question: we show that if $Aut(M)$ has no
compact quotients and $N$ is an $\omega$-categorical expansion of
$M$, then if $\Aut\left(N\right)$ has a finitely generated dense
subgroup, then so does $\Aut\left(M\right)$. More precisely, we show
that adding two elements from $\Aut\left(M\right)$ to $\Aut\left(N\right)$
yields a dense subgroup of $\Aut\left(M\right)$: see Theorem \ref{thm:If no CQ then finitely generated}.

In Section \ref{sec:A-topological-dynamics}, we give a dynamical
consequence of having a CIR. Our main motivation is to relate it to
the Ramsey property. We observe in Proposition \ref{prop:Ramsey has everything but over base}
that Ramsey structures admit a weaker form of independence relation.
We know by \cite{MR2140630} that the Ramsey property is equivalent
to extreme amenability of the automorphism group. It is then natural
to look for a dynamical interpretation of having a CIR, one goal being
to understand to what extent the Ramsey property is not sufficient
to imply it. We give a necessary condition for having a CIR which
sheds some light on this notion and can be used to prove negative
results.

The theorems presented in this paper lead to a number of open questions.
In particular, it would be interesting to understand the obstructions
to having a CIR and to prove more results about the automorphism groups
of structures with a CIR. It is also our hope that some ideas introduced
here could be used to develop a general theory of $\omega$-categorical
structures. One strategy we have in mind is to show that $\omega$-categorical
structures admit \emph{nice }expansions and then to prove \emph{relative}
statements which pull down properties from an expansion to the structure
itself. See Section \ref{sec:Further-questions} for some precise
conjectures.

We end this introduction by mentioning some previous work done on
this question: The existence of a cyclically dense conjugacy class
was shown for the random graph by Macpherson \cite{mac_automorphisms},
for the Urysohn space by Solecki \cite{Solecki2005}, for dense linear
orders by Darji and Mitchell \cite{Darji:2010} and recently for generic
posets by Glab, Gordinowicz and Strobin \cite{Glab2017}. Kechris
and Rosendal \cite{Kechris-Rosendal} study the property of having
a dense conjugacy class for Polish groups. They show that a number
of Polish groups admit cyclically dense conjugacy classes (see Theorem
2.10). Those include the group of homeomorphisms of the Cantor space
 and the automorphism group of a standard Borel space. Those groups
do not fit in our context, although it should be possible to generalize
our results so as to include them. In fact, the proof of Theorem 2.10
is very much in the same spirit as the proofs in this paper. Kwiatkowska
and Malicki \cite{Kwiatkowska2016} give sufficient conditions for
an automorphism group $G$ to have a cyclically dense conjugacy class,
which gives new examples such as structures with the free amalgamation
property and tournaments. Their conditions do not seem to formally
imply ours, but all the examples that they give (and in particular
structures with free amalgamation) are covered by our theorems. They
also show that under the same hypothesis $L_{0}(G)$ has a cyclically
dense conjugacy class. We did not study this.

\section{Preliminaries}

\subsection{\label{subsec:-cateogiricity,-Homogeneous-stru}$\omega$-categoricity,
Ultrahomogeneous structures, Fra\"iss\'e limits and model companions}

Here we recall the basic facts we need for this paper. 

Let $L$ be some countable first order language (vocabulary).

An $L$-theory is called $\omega$\emph{-categorical} if it has a
unique infinite countable model up to isomorphism. A countable model
$M$ is $\omega$-categorical if its complete theory $Th\left(M\right)$
is. By a theorem of Engeler, Ryll-Nardzewski and Svenonius, see e.g.,
\cite[Theorem 7.3.1]{Hod}, this is equivalent to saying that $G=\Aut\left(M\right)$
is \emph{oligomorphic}: for every $n<\omega$, there are only finitely
many orbits of the action of $G$ on $M^{n}$. It is also equivalent
to the property that every set $X\subseteq M^{n}$ which is invariant
under $G$ is $\emptyset$-definable in $M$.  Also, if $M$ is $\omega$-categorical
and $A$ is a finite subset of $M$ then $M_{A}$ is also $\omega$-categorical,
where $M_{A}$ is the expansion of $M$ for the language $L_{A}$
which adds a constant for every element in $A$. 

A countable $L$-structure $M$ is called \emph{ultrahomogeneous}
if whenever $f:A\to B$ is an isomorphism between two finitely generated
substructures $A,B$ of $M$, there is $\sigma\in\Aut\left(M\right)$
extending $f$. The \emph{age} of an $L$-structure $M$, $\Age\left(M\right)$,
is the class of all finitely generated substructures which can be
embedded into $M$. 

Recall that a class of finitely generated $L$-structures $K$ closed
under isomorphisms has\emph{ }the\emph{ hereditary property} (\emph{HP})
if whenever $A\in K$ and $B\subseteq A$ ($B$ is a substructure
of $A$), $B\in K$. The class $K$ has the \emph{joint embedding
property} (\emph{JEP}) if whenever $A,B\in K$ there is some $C$
such that both $A,B$ embed into $C$. It has the \emph{amalgamation
property} (\emph{AP}) if whenever $A,B,C\in K$ and $f_{B}:A\to B$,
$f_{C}:A\to C$ are embeddings, then there is some $D\in K$ and embeddings
$g_{B}:B\to D$, $g_{C}:C\to D$ such that $g_{B}\circ f_{B}=g_{C}\circ f_{C}$. 

We say that $K$ is\emph{ uniformly locally finite} if for some function
$f:\omega\to\omega$, for every $A\in K$ and $X$ a subset of $A$
of size $n$, the structure generated by $X$ has size $\leq f\left(\left|X\right|\right)$.
In the following, ``essentially countable'' means that $K$ contains
at most countably many isomorphism types of structures. 

We also recall the notions of model companions and model completions.
A theory $T$ is called model complete if whenever $M\subseteq N$
are models of $T$, $M\prec N$. Suppose that $T_{\forall}$ is a
universal theory. A theory $T'$ is the\emph{ model companion} of
$T_{\forall}$ if $T'$ is model complete and $T'_{\forall}=T_{\forall}$
(they have the same universal consequences). In other words, every
model of $T_{\forall}$ can be embedded in a model of $T'$. The theory
$T'$ is a \emph{model completion} of $T_{\forall}$ if in addition
it has elimination of quantifiers. Models companions are unique, if
they exist. For more, see \cite[Section 3.2]{TentZiegler} and \cite[Section 8.3]{Hod}. 
\begin{fact}
\label{fact:Fraisse limits}Suppose that that $K$ is an essentially
countable class of finite $L$-structures which has HP, JEP and AP. 
\begin{enumerate}
\item \cite[Theorem 7.1.2]{Hod} The class $K$ has a Fra\"iss\'e limit:
a \uline{unique} countable ultrahomogeneous model $M$ with the
same age. 
\item \cite[Theorem 7.4.1]{Hod} If $K$ is uniformly locally finite in
a finite language (or without assuming that $L$ is finite but instead
that for each $n<\omega$ there are finitely many isomorphism types
of structures generated by $n$ elements), then $M$ is $\omega$-categorical
and has quantifier elimination. (See remark below.)
\item If $T$ is a countable universal $L$-theory, $L$ is finite, $K$
is the class of finitely generated models of $T$ and $K$ is uniformly
locally finite (or without assuming that $L$ is finite but instead
that $K$ has at most finitely many isomorphism types of structures
generated by $n$ elements for all $n<\omega$), then the theory $Th\left(M\right)$
is $\omega$-categorical and is the model completion of $T$. (See
remark below.)
\item \cite[Theorem 7.1.7]{Hod}The converse to (1) also holds: if $M$
is an ultrahomogeneous then the age of $M$ satisfies HP, JEP and
AP. 
\end{enumerate}
\end{fact}

\begin{rem}
The assumptions in parenthesis in (2) is not stated in \cite[Theorem 7.4.1]{Hod},
but a straightforward modification of that proof gives this. 

We could not find an explicit reference for (3) (which is well-known),
so here is a short argument. Since $Th\left(M\right)$ eliminates
quantifiers by (2), it is enough to show that $Th\left(M\right)_{\forall}=T_{\forall}=T$
(since $T$ is universal). 

If $\psi$ is universal and $T\models\psi$, then since $\Age\left(M\right)\subseteq K$,
$M\models\psi$. On the other hand, if $M\models\psi$ where $\psi$
is universal, and $T\not\models\psi$, then there is a model $A'\models T$
such that $A'\models\neg\psi$, and since $\neg\psi$ is existential,
the same is true for some finitely generated model $A\subseteq A'$,
so $A\in K$ and since $A\in\Age\left(M\right)$, we get a contradiction. 
\end{rem}

Classes $K$ as in Fact \ref{fact:Fraisse limits} are called Fra\"iss\'e
classes or amalgamation classes.

Some examples of Fra\"iss\'e limits include DLO (dense linear order),
i.e., $Th\left(\Qq,<\right)$ (here we identify the Fra\"iss\'e
limit and its theory), the random graph, the random poset (partially
ordered set), the random tournament, and more. One example that we
will be interested in is that of dense trees. 
\begin{example}
\label{exa:dense trees}Let $L_{dt}=\left\{ <,\wedge\right\} $, and
let $T_{dt,\forall}$ be the universal theory of trees with a meet
function $\wedge$. Then $T_{dt,\forall}$ has an $\omega$-categorical
model completion by Fact \ref{fact:Fraisse limits} (note that the
tree generated by a finite set $B$ is just $B\cup\set{x\wedge y}{x,y\in B}$).
We denote the model companion by $\DenseTrees$ and call the unique
countable model the \emph{dense tree}. See also \cite[Section 2.3.1]{simon2015guide}. 
\end{example}

Recall that a structure $M$ is \emph{homogeneous} if whenever $a,b$
are finite tuples of the same length, and $a\equiv b$ (which means
$\tp\left(a/\emptyset\right)=\tp\left(b/\emptyset\right)$, i.e.,
the tuples $a,b$ have the same type), then there is an automorphism
taking $a$ to $b$. Note that ultrahomogeneous structures are homogeneous
and the same is true for $\omega$-categorical ones. When $M$ is
homogeneous, an \emph{elementary map} $f:A\to B$ for $A,B\subseteq M$
is just a restriction of an automorphism of $M$. 

Finally, we use $\C$ to represent a monster model of the appropriate
theory. This is a big saturated (so also homogeneous) model that contains
all the models and sets we will need. This is standard in model theory.
For more, see \cite[Section 6.1]{TentZiegler}. 

\subsection{\label{subsec:A-mix-of}A mix of two Fra\"is\'e limits}

Suppose that $K_{1},K_{2}$ are two amalgamation classes of finite
structures in the languages $L_{1},L_{2}$ respectively. Assume the
following properties:
\begin{enumerate}
\item The symmetric difference $L_{1}\mathrela{\triangle}L_{2}$ is relational. 
\item The class $K$ of finite $L_{1}\cup L_{2}$-structures $A$ such that
$A\restriction L_{1}\in K_{1}$ and $A\restriction L_{2}\in K_{2}$
is an amalgamation class. Let $M$ be its Fra\"is\'e limit.
\item For every $A\in K_{1}$, there is some expansion $A'$ of $A$ to
an $L_{1}\cup L_{2}$-structure such that $A'\restriction L_{2}\in K_{2}$,
and similarly that for every $B\in K_{2}$ there is some expansion
$B'$ to $L_{1}\cup L_{2}$ whose restriction to $L_{1}$ is in $K_{1}$. 
\item If $A\in K$ and $A\restriction L_{1}\subseteq B\in K_{1}$ then there
is an expansion $B'$ of $B$ to $L_{1}\cup L_{2}$ such that $A\subseteq B'$
and $B'\restriction L_{2}\in K_{2}$, and similarly for $L_{2}$. 
\end{enumerate}
Under all these conditions we have the following.
\begin{prop}
\label{prop:restriction of a mix of two Faisse} The structure $M\restriction L_{1}=M_{1}$
is the Fra\"is\'e limit of $K_{1}$ and $M\restriction L_{2}=M_{2}$
is the Fra\"is\'e limit of $K_{2}$.
\end{prop}

\begin{proof}
Start with $M_{1}$ (for $M_{2}$, the proof is the same). It is enough
to show that $M_{1}$ is ultrahomogeneous and that $\Age\left(M_{1}\right)=K_{1}$.
The second statement follows from (2), (3) above. For the first, by
\cite[Lemma 7.1.4]{Hod} it is enough to show that if $A\subseteq B$
are from $K_{1}$ and $f:A\to M_{1}$ then there is some $g:B\to M_{1}$
extending $f$. Using $f$ we can expand $A$ to an $L_{1}\cup L_{2}$-structure
$A'$ in such a way that $f$ is an embedding to $M$ (this uses the
fact that $L_{2}\backslash L_{1}$ is relational). By (4) we can expand
$B$ to an $L_{1}\cup L_{2}$-structure $B'$ such that $A'\subseteq B'$
and $B'\in K$. Since $M$ is the Fra\"iss\'e limit of $K$, it
follows by \cite[Lemma 7.1.4]{Hod} again that $f$ can be extended
to $g:B'\to M$, and in particular, $g\restriction L_{1}$ is the
embedding we seek. 
\end{proof}

\subsection{\label{subsec:The-maximal-compact}The maximal compact quotient of
the automorphism group of a countable $\omega$-categorical structure}

Assume in this section that $M$ is $\omega$-categorical and countable,
and let $G=\Aut\left(M\right)$, considered as a topological group
in the product topology. The contents of this section are folklore
but we give the details for the sake of readability. Recall that for
a structure $M$ and $A\subseteq M$, $\acl\left(A\right)$ is the
set of all algebraic elements over $A$ (elements satisfying an algebraic
formula over $A$: one with finitely many solutions). Similarly, $\dcl\left(A\right)$
is the set of all elements definable over $A$. In the context of
$\omega$-categorical structures, $\acl\left(A\right)$ and $\dcl\left(A\right)$
are defined in terms of the size of the orbit of the action of $G$
fixing $A$ being finite or a singleton respectively. In the next
proposition, we describe the maximal compact quotient of $\Aut\left(M\right)$
in model theoretic terms. This uses the notion of $M^{\eq}$: the
expansion of $M$ obtained by adding a new sort for every $\emptyset$-definable
quotient of some $\emptyset$-definable set. See \cite[Section 8.4]{TentZiegler}
for more. 

For $A\subseteq M$, $\Aut\left(M/A\right)$ is the group of automorphisms
of $M$ fixing $A$. This is a closed subgroup of $G$ which is normal
when $A$ is invariant under $\Aut\left(M\right)$, thus the quotient
$G/\Aut\left(M/A\right)$ is a Hausdorff topological group (with the
quotient topology). We identify $\Aut\left(M\right)$ and $\Aut\left(M^{\eq}\right)$
so that we can put $A=\acl^{\eq}\left(\emptyset\right)$. In this
case, it is also compact as the next proposition says. 
\begin{prop}
\label{prop:Fixing acl(0) is compact quotient in omega-cat} The group
$G/\Aut\left(M/\acl^{\eq}\left(\emptyset\right)\right)$ is a compact
Hausdorff (in fact \textemdash{} profinite) group. 
\end{prop}

\begin{proof}
Let $H$ be the group of all elementary maps from $\acl^{\eq}\left(\emptyset\right)$
to $\acl^{\eq}\left(\emptyset\right)$, also denoted by $\Aut\left(\acl^{\eq}\left(\emptyset\right)\right)$.
The group $H$ is naturally profinite as an inverse system of the
family 
\[
\set{H_{X}}{X\subseteq M^{\eq},\text{ a finite }\emptyset\text{-definable set}},
\]
 where $H_{X}$ is the group of elementary permutations of $X$. 

Let $\res:G\to H$ be the restriction map $\res\left(\sigma\right)=\sigma\restriction\acl^{\eq}\left(\emptyset\right)$.
We will show that $\res$ is onto. Using back-and-forth, it is enough
to show that given any complete type $p\left(x\right)$ for $x$ in
the home sort over $A\cup\acl^{\eq}\left(\emptyset\right)$ where
$A\subseteq M$ is finite, $p$ can be realized in $M$.  We work
in the monster model $\C$. Let $E=\mathordi{\equiv_{A\cup\acl^{\eq}\left(\emptyset\right)}}$
be the equivalence relation of having the same type over $A\cup\acl^{\eq}\left(\emptyset\right)$.
Then $E$ is $A$-invariant and has boundedly many classes in $\C$.
By $\omega$-categoricity, as $A$ is finite $E$ is definable over
$A$, and by compactness, $E$ has finitely many classes. But then
for every $E$-class there must be a representative in $M$. Since
a realization of $p$ must have an $E$-equivalent element in $M$,
$p$ is realized in $M$. Note that by compactness we get that every
such $p$ is isolated by its restriction to $A\cup X$ where $X$
is some finite $\emptyset$-definable set in $M^{\eq}$.  

The kernel of $\res$ is precisely $G^{0}=\Aut\left(M/\acl^{\eq}\left(\emptyset\right)\right)$,
so $\res$ induces an isomorphism of groups $G/G^{0}\to H$. The group
$G/G^{0}$ is also a topological group when equipped with the quotient
topology. This map is easily seen to be continuous. To see that it
is open, it is enough to show that the image of an open neighborhood
$V$ of $\id\cdot G^{0}$ in $G/G^{0}$ contains an open neighborhood
of $\id$ in $H$. The preimage of $V$ in $G$ is some open set
$U\subseteq G$ containing $\id$. Suppose $\id\in U_{b}=\set{\sigma\in G}{\sigma\left(b\right)=b}\subseteq U$
is some basic open set. As we noted above, there is some finite $\emptyset$-definable
set $X\subseteq M^{\eq}$ such that $\tp\left(b/\acl^{\eq}\left(\emptyset\right)\right)$
is isolated by $\tp\left(b/X\right)$. Then if $\tau\in G$ is such
that $\tau\restriction X=\id_{X}$, then there is some $\sigma\in\Aut\left(M/\acl^{\eq}\left(\emptyset\right)\right)$
such that $\sigma\tau\left(b\right)=b$, so $\sigma\tau\in U_{b}$,
but then  $\tau\in U$ (because $U$ is a union of cosets of $\Aut\left(M/\acl^{\eq}\left(\emptyset\right)\right)$).
Hence, the image of $V$ contains the open set $\set{\tau\in H}{\tau\restriction X=\id_{X}}$. 

Together these two groups are isomorphic as topological groups, so
are profinite. 
\end{proof}
\begin{defn}
We let $G^{0}=\Aut\left(M/\acl^{\eq}\left(\emptyset\right)\right)$. 
\end{defn}

\begin{prop}
\label{prop:Contains every compact}If $H\trianglelefteq G$ is normal
and closed and $G/H$ is compact then $G^{0}\leq H$. 
\end{prop}

\begin{proof}
 Let $E_{n}$ be the equivalence relation on $M^{n}$ of having the
same $H$-orbit. Then $E_{n}$ refines $\equiv$ (having the same
type over $\emptyset$) and is definable in $M$. Indeed, suppose
that $\sigma\in G$. Then for every $a\in M^{n}$, $\sigma O_{H}\left(a\right)=O_{H}\left(\sigma\left(a\right)\right)$
(because $H$ is normal), where $O_{H}\left(a\right)$ denotes the
orbit of $a$ under $H$. Hence $E_{n}$ is invariant under $G$ so
$\emptyset$-definable. 

In addition, $G$ acts transitively on the $H$-orbits within each
$\equiv$-class by $\sigma\cdot O_{H}\left(a\right)=O_{H}\left(\sigma\left(a\right)\right)$.
The stabilizer of $O_{H}\left(a\right)$ is $\set{\sigma\in G}{\sigma\left(O_{H}\left(a\right)\right)=O_{H}\left(a\right)}$
so open (if $\sigma$ is there, then to make sure that $\sigma'$
is there, it is enough that $\sigma'\left(a\right)=\sigma\left(a\right)$).
Note that this action factors through $H$ (i.e., the action $\sigma H\cdot O_{H}\left(a\right)=O_{H}\left(\sigma\left(a\right)\right)$
is well-defined). Hence, the stabilizer is open in $G/H$ and hence
has finite index in $G/H$, so the number of orbits of $H$ under
this action in every $\equiv$-class is finite. By $\omega$-categoricity,
the number of $\equiv$-classes (of $n$-tuples) is finite, so the
number of orbits of $H$ is finite. 

In summary, $E_{n}$ is definable and has finitely many classes. Hence
these classes belong to $\acl^{\eq}\left(\emptyset\right)$. Given
$\sigma\in G^{0}$, $\sigma$ fixes the orbits of $H$ under its action
on $M^{n}$. As $H$ is closed, this means that $\sigma\in H$.  
\end{proof}
\begin{cor}
\label{cor:G0 the maximal compact quotient}The group $G/G^{0}$ is
the maximal compact Hausdorff quotient of $G$. 
\end{cor}

In light of Corollary \ref{cor:G0 the maximal compact quotient},
$G$ \emph{has no compact quotients} (by which we mean that there
is no nontrivial compact Hausdorff group which is an image of $G$
under a continuous group homomorphism) iff $G^{0}=G$. If $N$ is
a normal closed subgroup of $G$, then we would like to say that $\left(G,N\right)$
has no compact quotients iff $G/N$ has no compact quotients. Let
us generalize this to any closed subgroup. 
\begin{defn}
\label{def:a pair having a CQ}Suppose that $H\leq G$ is closed.
We will say that the pair $\left(G,H\right)$ has \emph{no compact
quotients} if for all $g\in G$, there is some $h\in H$ such that
$g\restriction\acl^{\eq}\left(\emptyset\right)=h\restriction\acl^{\eq}\left(\emptyset\right)$. 
\end{defn}

\begin{prop}
\label{prop:no compact quotients}Suppose that $H\leq G$ is closed.
Then $\left(G,H\right)$ has no compact quotients iff for every closed
and normal $N\trianglelefteq G$ such that $G/N$ is compact, $NH=G$. 
\end{prop}

\begin{proof}
First note that $\left(G,H\right)$ has no compact quotient iff $\set{gH}{g\in G}=\set{gH}{g\in G^{0}}$.
This happens iff $G=G^{0}H$. Hence the direction from right to left
follows by taking $N=G^{0}$. The direction from left to right is
immediate by Proposition \ref{prop:Contains every compact}. 
\end{proof}
\begin{cor}
If $H\leq G$ is closed and normal then $\left(G,H\right)$ has no
compact quotients iff $G/H$ has no compact quotients as a topological
group (i.e., there is no nontrivial compact Hausdorff quotient).
\end{cor}

\begin{proof}
Left to right: suppose that $G/H$ has a compact quotient. Then there
is a normal closed subgroup $N\trianglelefteq G$ such that $H\leq N$
and $G/N$ is compact. By Proposition \ref{prop:Contains every compact},
$N$ contains $G^{0}$. Thus, $G^{0}H\leq N$ and hence $G=N$ (by
Proposition \ref{prop:no compact quotients}). 

Right to left: for a finite $\emptyset$-definable subset $X\subseteq\acl^{\eq}\left(\emptyset\right)$,
let $G_{X}^{0}=\Aut\left(M/X\right)\leq G$ (where we identify $G$
with $\Aut\left(M^{\eq}\right)$). Note that as $X$ is definable,
$G_{X}^{0}$ is normal, hence so is the product with $H$. Since
$\left[G:G_{X}^{0}\right]$ is finite, $HG_{X}^{0}$ is closed as
a finite union of translates of $G_{X}^{0}$ so that $G/HG_{X}^{0}$
is a compact (even finite) Hausdorff quotient of $G/H$, so it must
be trivial and hence that $HG_{X}^{0}=G$. Take $g\in G$, we need
to find $h\in H$ such that $g\restriction\acl^{\eq}\left(\emptyset\right)=h\restriction\acl^{\eq}\left(\emptyset\right)$.
By what we just said, we have that ({*}) for every finite definable
$X\subseteq\acl^{\eq}\left(\emptyset\right)$ there is some $h_{X}\in H$
such that $g\restriction X=h_{X}\restriction X$. But since every
finite $X\subseteq\acl^{\eq}\left(\emptyset\right)$ is contained
in a finite $\emptyset$-definable set $X'\subseteq\acl^{\eq}\left(\emptyset\right)$
($X'$ is just the union of all conjugates of $X$), ({*}) is true
for all finite subset $X\subseteq\acl^{\eq}\left(\emptyset\right)$. 

For every finite tuple $a$ from $M$, $O_{H}\left(a\right)$ (the
orbit of $a$ under $H$) is $a$-definable (because $H$ is normal,
$g\cdot O_{H}\left(a\right)=O_{H}\left(g\left(a\right)\right)$ for
any $g\in G$, so that $O_{H}\left(a\right)$ is $a$-invariant thus
$a$-definable by $\omega$-categoricity). In $M^{\eq}$, every definable
set $X$ has a code $\code X\in M^{\eq}$ (such that the automorphisms
fixing $X$ setwise in $M$ are precisely the automorphisms fixing
$\code X$). (This notation is a bit misleading since there could
be many possible codes for $X$.) Let $D\subseteq M^{\eq}$ be the
collection of all possible codes $\code{O_{H}\left(a\right)}$ for
all finite tuples $a$ from $M$. Then $D$ is invariant under $G$
since $g\left(\code{O_{H}\left(a\right)}\right)=\code{O_{H}\left(g\left(a\right)\right)}$
for all $g\in G$ (i.e., $g\left(\code{O_{H}\left(a\right)}\right)$
is a code for $O_{H}\left(g\left(a\right)\right)$). Let $\bar{c}$
be a tuple enumerating $\acl^{\eq}\left(\emptyset\right)$. Since
every $h\in H$ fixes $D$ pointwise, ({*}) gives us that $\bar{c}\equiv_{D}g\left(\bar{c}\right)$
(because to check this equation it is enough to consider finite subtuples).
Thus, the map $f$ taking $\bar{c}$ to $g\left(\bar{c}\right)$ fixing
$D$ is an elementary map. By a back-and-forth argument almost identical
to the one given in the proof of Proposition \ref{prop:Fixing acl(0) is compact quotient in omega-cat},
there is some automorphism $h\in G$ extending $f$ (the point is
that the relation $\equiv_{\acl^{\eq}\left(\emptyset\right)\cup D\cup A}$
is bounded and $A$-invariant for any finite set $A$, hence definable
and hence has finitely many classes, all of them realized in $M$).
Since $H$ is closed, and $h$ fixes all $H$-orbits setwise (as it
fixes $D$), $h\in H$. Finally, $h\restriction\acl^{\eq}\left(\emptyset\right)=g\restriction\acl^{\eq}\left(\emptyset\right)$
as requested. 
\end{proof}
\begin{example}
\label{exa:If G0=00003DG, then no compact quotients}If $M'$ is an
expansion of $M$ and $\acl^{\eq}\left(\emptyset\right)=\dcl^{\eq}\left(\emptyset\right)$
in $M$ then $\left(G,\Aut\left(M'\right)\right)$ has no compact
quotients. This is because in that case, $G^{0}=G$. 
\end{example}

\subsection{\label{subsec:Expansions-and-reducts}Expansions and reducts of $\omega$-categorical
structures}

A group $H$ acts \emph{oligomorphically} on a set $X$ if for all
$n<\omega$, the number of orbits of $X^{n}$ under the action of
$H$ is finite for every $n<\omega$.

If $M$ is countable and $\omega$-categorical, and $H\leq G$ is
closed, then $H=\Aut\left(M'\right)$ for some expansion of $M$.
In addition, if $H$ acts oligomorphically on $M$, then $M'$ is
$\omega$-categorical by Ryll-Nardzewski. On the other hand, if $G\leq H$
where $H$ is a closed subgroup of the group of permutations of $M$,
then $H=\Aut\left(M'\right)$ for some ($\omega$-categorical) reduct
$M'$ of $M$. Two such reducts $M'$, $M''$ are the same up to bi-definability
if they have the same definable sets, which is equivalent to $\Aut\left(M'\right)=\Aut\left(M''\right)$.
\begin{prop}
\label{prop:G0 acts oligomorphically and has no CQ}Let $M$ be $\omega$-categorical
and $G=\Aut\left(M\right)$. Then $G^{0}\leq G$ acts oligomorphically
on $M$ and $G^{0}=\Aut\left(M'\right)$ for an $\omega$-categorical
expansion $M'$ of $M$ with no compact quotients. 
\end{prop}

\begin{proof}
Let $M'_{0}$ be the expansion of $M^{\eq}$ obtained by naming (i.e.,
adding constants for) every element in $\acl^{\eq}\left(\emptyset\right)$.
Then $M_{0}'$ is still $\omega$-categorical (as a many sorted structure)
since for any given sort (or finite collection of sorts) $S$, the
equivalence relation $\equiv_{\acl^{\eq}\left(\emptyset\right)}$
on $S$-tuples, of having the same type over $\acl^{\eq}\left(\emptyset\right)$,
is bounded, definable and hence finite, as in the arguments above.
Let $M'$ be the reduct to the home sort (so it is also $\omega$-categorical).
By definition, $G^{0}=\Aut\left(M'\right)$. Moreover, letting $H=\Aut\left(M'\right)$,
we have that $H^{0}=H$. This is because $\acl^{\eq}\left(\emptyset\right)=\acl^{\eq}\left(\acl^{\eq}\left(\emptyset\right)\right)$. 
\end{proof}

\subsection{\label{subsec:Ramsey, top dynamics}A discussion of Ramsey classes
and topological dynamics}

\subsubsection{Ramsey Classes}

Let us start with the definition.
\begin{defn}
For two $L$-structures $A,B$, we let ${B \choose A}$ be the set
substructures of $B$ isomorphic to $A$. Suppose that $K$ is a class
of finite $L$-structures. We say that $K$ is a \emph{Ramsey class}
if for every $A,B\in K$ and $k<\omega$ there is some $C\in K$ such
that $C\to\left(B\right)_{k}^{A}$: for every function $f:{C \choose A}\to k$
there is some $B'\in{C \choose B}$ such that $f$ is constant on
${B' \choose A}$. 

Say that an ultrahomogeneous $L$-structure $M$ with a quantifier-free
definable linear order is a\emph{ Ramsey structure} if $\Age\left(M\right)$
is a Ramsey class.
\end{defn}

Ramsey classes are extremely important classes of finite structure.
There are many examples of Ramsey classes, in particular the class
of finite linear orders (this is just Ramsey's theorem) and furthermore,
by a theorem of Ne\v{s}et\v{r}il and R\"odl \cite[Theorem A]{MR692827},
proved independently by Harrington and Abramson \cite[Appendix B]{MR503795},
the class of all finite linearly ordered graphs, or more generally
the class of all finite linearly ordered structures in a fixed finite
relational language is Ramsey. In fact, \cite[Theorem 4.26]{Hubicka2016}
generalizes this to allow function symbols as well. 

There is another definition of Ramsey structures that colors embeddings
instead of copies, see \cite[Definition 2.2]{MR3497266}. This is
equivalent to our definition since we asked for a quantifier-free
definable linear order (so that finite substructures are rigid). If
we drop the requirement that there is a definable linear order then
these definitions do not agree in general.  In fact, using the alternative
definition there must be a definable order in the $\omega$-categorical
case, so these are equivalent in this case. 
\begin{fact}
\cite[Corollary 2.26]{MR3497266} If $M$ is an $\omega$-categorical
ultrahomogeneous Ramsey structure according to \cite[Definition 2.2]{MR3497266},
then there is a definable linear order on $M$. 
\end{fact}

What about dense trees? Adding a generic linear order to a dense tree
will not result in a Ramsey structure. By this we mean the model completion
of the theory $T_{dt,<,\forall}$ in the language $\left\{ <,\wedge,<'\right\} $
which says that the $\left\{ <,\wedge\right\} $-part is a meet tree,
and $<'$ is a linear order. The class of finite structures of $T_{dt,<,\forall}$
easily has HP, JEP and HP, thus this model completion exists (see
Fact \ref{fact:Fraisse limits}). Call its Fra\"iss\'e limit the
generically linearly ordered tree. It turns out that this structure
is not Ramsey, see Claim \ref{claim:The-gen lin tree is not Ramsey}.
In any case, we can add a linear order to the tree structure and make
it Ramsey.
\begin{example}
\label{exa:lex trees are ramsey} (\cite[Corollary 3.17]{MR3373612},
and see there for more references) Let $L=\left\{ <,<_{\lex},\wedge\right\} $
and let $M$ be the $L$-structure whose universe is the tree $\omega^{<\omega}$
with the natural interpretations of $<$ as the tree order, $\wedge$
as the meet function ($s\wedge t=s\restriction\len\left(s\wedge t\right)$
where $\len\left(s\wedge t\right)=\max\set k{s\restriction k=t\restriction k}$),
$<_{\lex}$ as the lexicographical order ($s<_{\lex}t$ iff $s<t$
or $s\left(\len\left(s\wedge t\right)\right)<t\left(\len\left(s\wedge t\right)\right)$).
Then $M$ is a Ramsey structure. 
\end{example}

\begin{fact}
\label{fact:Ramsey classes have AP} \cite[Theorem 4.2]{MR2128088}
A Ramsey class that has HP and JEP has AP (see Section \ref{subsec:-cateogiricity,-Homogeneous-stru}
for the definitions). 
\end{fact}

It is easy to see that $K=\Age\left(\omega^{<\omega}\right)$ in the
language $\left\{ <,<_{\lex},\wedge\right\} $ has JEP, and thus
we conclude that it has AP. Let $M$ be its Fra\"iss\'e limit and
let $T_{dt,<_{lex}}=Th\left(M\right)$. Since $K$ is uniformly locally
finite  it follows that $M$ is $\omega$-categorical and $T_{dt,<_{lex}}$
has quantifier elimination (see Fact \ref{fact:Fraisse limits}).

By Proposition \ref{prop:restriction of a mix of two Faisse} we have
that the restriction of $M$ to $\left\{ <,\wedge\right\} $ is a
dense tree, i.e., a model of $\DenseTrees$ (see Example \ref{exa:dense trees}).
Here, $L_{1}=\left\{ <,\wedge\right\} $, $L_{2}=\left\{ <,<_{\lex},\wedge\right\} $,
$K_{1}$ the class of finite meet trees and $K_{2}=\Age\left(\omega^{<\omega}\right)$.
Similarly, the restriction of the generically linearly ordered tree
is also a dense tree. 
\begin{claim}
\label{claim:The-gen lin tree is not Ramsey}The generically linearly
ordered tree is not a Ramsey structure. 
\end{claim}

\begin{proof}
Let $K=\Age\left(M\right)$ where $M$ is the countable generically
linearly ordered tree. As $N=M\restriction\left\{ <,\wedge\right\} \models\DenseTrees$,
by $\omega$-categoricity, there is some linear order $<_{lex}$ such
that $\left(N,<_{lex}\right)\models T_{dt,<_{lex}}$. 

Fix some $A\in K$ whose universe contains 3 elements $a,b,a\wedge b$
such that $a\wedge b<a,b$ and $a<'b<'a\wedge b$. Define a coloring
$f:{M \choose A}\to2$ by $f\left(A'\right)=0$ iff {[}$a<'b$ iff
$a<_{lex}b$ (in $A'$){]}. If $M$ were Ramsey, there would be some
homogeneous $B'\in{M \choose B}$ where $B\in K$ is such that $B$
contains 5 elements $a,b,a\wedge b,c,a\wedge c$ where $a\wedge c<a,c$
and $a\wedge c<a\wedge b<a,b$ and $a<'c<'b<'a\wedge b<'b\wedge c$.
It follows that for any copy of $B$ in $M$, $a<_{lex}c$ iff $b<_{lex}c$.
However in $B'$ we have that $a<_{lex}c$ iff $c<_{lex}b$ \textemdash{}
contradiction. 
\end{proof}
We end this discussion with the following fact.
\begin{fact}
\label{fact:Expanding Ramsey by naming elements}\cite[Theorem 3.10]{MR3497266}
If $M$ is an ultrahomogeneous Ramsey structure, and $c\in M$, then
the structure $M'=\left(M,c\right)$ where $c$ is a named constant
is still Ramsey (and ultrahomogeneous). 
\end{fact}

\subsubsection{Topological dynamics and extremely amenable groups}

Let us first recall some basic notions from topological dynamics. 

Suppose that $G$ is a topological group. A \emph{$G$-flow} is a
compact Hausdorff space $X$ with a continuous action of $G$. A subflow
of $X$ is a compact subspace $Y\subseteq X$ that is preserved by
the action of $G$, i.e., $gY=Y$ for all $g\in G$. A \emph{$G$-ambit}
is pair $\left(X,x_{0}\right)$ where $X$ is a $G$-flow and $x_{0}$
has a dense orbit. A universal $G$-ambit is a $G$-ambit $\left(X,x_{0}\right)$
such that for any ambit $\left(Y,y_{0}\right)$ there is a map $f:X\to Y$
taking $x_{0}$ to $y_{0}$ that commutes with the action: $gf\left(x\right)=f\left(gx\right)$
for all $x\in X$ (it follows that $f$ is onto). A universal $G$-ambit
exists and is unique (see \cite[Chapter 8]{MR956049}). Finally, $G$
is called \emph{extremely amenable} if for every $G$-flow $X$, there
is some fixed point $x\in X$ (i.e., $gx=x$ for all $g\in G$). 

Kechris, Pestov, and Todorcevic \cite{MR2140630} found a striking
link between Ramsey classes and topological dynamics, described in
the following theorem.
\begin{fact}
\label{fact:ramsey iff extremeley amenable}\cite[Theorem 4.7]{MR2140630}
Suppose that $M$ is a countable ultrahomogeneous linearly ordered
structure in a countable language. Then $\Aut\left(M\right)$ is extremely
amenable iff $M$ is a Ramsey structure. 
\end{fact}

\section{Having a cyclically dense conjugacy class}
\begin{defn}
Suppose that $G$ is a topological group.
\begin{enumerate}
\item The group $G$ has\emph{ finite topological rank }if it has a finitely
generated dense subgroup. Similarly, $G$ has\emph{ topological rank
$n$ }(or \emph{topologically $n$-generated}) if there are $\set{f_{i}}{i<n}\subseteq G$
which generate a dense subgroup.
\item The group $G$ has a \emph{cyclically dense conjugacy class} if there
are $f_{1},f_{2}\in G$ such that $\set{f_{1}^{-n}f_{2}f_{1}^{n}}{n\in\Zz}$
is dense in $G$. 
\end{enumerate}
\end{defn}

\begin{rem}
\label{rem:constraint for cdcc compact quotient} If $f:G_{1}\to G_{2}$
is a surjective continuous homomorphism, and $G_{1}$ has a dense
conjugacy class, then so does $G_{2}$. Also, if $H$ is a finite
nontrivial group (with the discrete topology), then $H$ cannot have
a dense conjugacy class. Therefore, the same is true for nontrivial
profinite groups. Hence if $G$ is a topological group with a nontrivial
profinite quotient, it does not contain a dense conjugacy class. Let
$M$ be countable and $\omega$-categorical and $G=\Aut\left(M\right)$.
By Proposition \ref{prop:Fixing acl(0) is compact quotient in omega-cat},
$G/G^{0}$ is profinite. It follows that one constraint against having
a dense conjugacy class is having a nontrivial compact quotient. In
model theoretic terms, it means that if $\acl^{\eq}\left(\emptyset\right)\neq\dcl^{\eq}\left(\emptyset\right)$
(equivalently, $G^{0}\neq G$), then $G$ cannot have a dense conjugacy
class.

Moreover, in general, we can have that $\left(\left(\Zz/2\Zz\right)^{\omega},+\right)$
is a quotient of $\Aut\left(M\right)$, which is locally finite so
certainly not topologically finitely generated, in which case $G=\Aut\left(M\right)$
cannot be topologically finitely generated. For example, let $L=\set{E_{n}}{n<\omega}$
where each $E_{n}$ is a $2n$-ary relation. Let $T_{\forall}$ say
that $E_{n}$ is an equivalence relation with two classes, and that
$\left(x_{1},\ldots,x_{n}\right)\mathrela{E_{n}}\left(y_{1},\ldots,y_{n}\right)\to\bigwedge_{i\neq j}x_{i}\neq x_{j}\land y_{i}\neq y_{j}$
(there is no relation between different $E_{n}$'s). The class of
finite $T_{\forall}$ models has AP and JEP (and it is essentially
countable), and by Fact \ref{fact:Fraisse limits} (which we can use
since for any $n<\omega$ there are finitely many isomorphism types
of $n$-element structures), there is a model completion $T$ which
is $\omega$-categorical.  Let $M$ be the countable model. Then
$\acl^{\eq}\left(\emptyset\right)$ contains $M^{n}/E_{n}$ for all
$n<\omega$, and for any $\eta\in\Zz/2\Zz$, there is an automorphism
$\sigma\in G=\Aut\left(M\right)$ such that $\sigma\restriction M^{n}/E_{n}$
is the identity iff $\eta\left(n\right)=0$. In fact one can show
that $G/G^{0}=\left(\left(\Zz/2\Zz\right)^{\omega},+\right)$ and
that $\sigma$ fixes $\acl^{\eq}\left(\emptyset\right)$ iff it fixes
all $E_{n}$-classes.    

This construction is due to Cherlin and Hrushovski (see \cite[Proof of Theorem 5.2]{MR1015304}
and \cite[Addendum (2)]{Lascar}). Using a similar technique, in \cite[Lemma 3.1]{MR1042609}
it is shown that any profinite group $H$ which has a countable basis
of open subgroups can be realized as $G/G^{0}$ for some automorphism
group $G$ of an $\omega$-categorical structure $M$.  
\end{rem}

\begin{defn}
Suppose that $M$ is some structure and $a,b\in M$ are some tuples.
We write $a\ind^{ns}b$ to say that $\tp\left(a/b\right)$ does not
split over $\emptyset$. When $M$ is homogeneous, this means, letting
$B$ be the set $b$ enumerates: if $g:B'\to B''$ is a partial automorphism
of $B$ (i.e., $B',B''\subseteq B$ and $g$ extends to an automorphism
of $M$) then $g$ extends to an automorphism of $M$ which fixes
$a$ pointwise. 
\end{defn}

In the next definition, our convention is that for sets $A,B$, we
write $A\ind B$ if this is true for tuples enumerating $A,B$.
\begin{defn}
An automorphism $\sigma\in\Aut\left(M\right)$ is \emph{repulsive}
if for every finite set $A\subseteq M$ there is some $n$ such that
$A\ind^{ns}\sigma^{n}\left(A\right)$ and $\sigma^{n}\left(A\right)\ind^{ns}A$.
Say that $\sigma$ is \emph{strongly repulsive} if this is true for
all $m\geq n$ as well. 
\end{defn}

Suppose that $M$ is some structure. For $k<\omega$, add predicates
$P_{1},\ldots,P_{k}$ to the language, and let $\bigsqcup_{k}M$ be
the disjoint union of $k$ copies of $M$, one for each predicate,
where each copy has the same structure as $M$. Then $\Aut\left(\bigsqcup_{k}M\right)=\Aut\left(M\right)^{k}$. 
\begin{prop}
\label{prop:product}If $\sigma$ is a (strongly) repulsive automorphism
of an $L$-structure $M$, then $\sigma^{\times k}\in\Aut\left(M\right)^{k}$is
a (strongly) repulsive automorphism of the structure $\bigsqcup_{k}M$
for all $k<\omega$. 
\end{prop}

\begin{proof}
Suppose that $A\subseteq\bigsqcup_{k}M$ is finite. Then we may assume,
enlarging $A$, that $A=\bigsqcup_{k}A_{0}$ for some finite $A_{0}\subseteq M$
(i.e., the disjoint union of the same set in the different predicates).
Thus the proposition follows from the fact that if $A_{0}\ind^{ns}B_{0}$
in $M$, then $\bigsqcup_{k}A_{0}\ind^{ns}\bigsqcup_{k}B_{0}$ in
$\bigsqcup_{k}M$, which is clear. 
\end{proof}
A repulsive automorphism is a special case of a topologically transitive
map:
\begin{defn}
Suppose that $X$ is a topological space. A map $f:X\to X$ is called
\emph{topologically transitive} if for every two nonempty open sets
$U,V\subseteq X$, there is some $n<\omega$ such that $f^{n}\left(U\right)\cap V\neq\emptyset$. 
\end{defn}

\begin{lem}
Suppose that $M$ is a countable structure and that $\sigma\in G=\Aut\left(M\right)$
is repulsive. Then conjugation by $\sigma$ in $G$ is topologically
transitive.
\end{lem}

\begin{proof}
Denote by $f:G\to G$ the conjugation by $\sigma$. Suppose that $U,V$
are two nonempty basic open subsets of $G$, i.e., $U=U_{a,b}=\set{\tau\in G}{\tau\left(a\right)=b}$
(where $a,b$ are finite tuples) and $V=U_{c,d}$. Note that $f\left(U_{a,b}\right)=U_{\sigma\left(a\right),\sigma\left(b\right)}$.
So we need to find some $n<\omega$ such that $U_{\sigma^{n}\left(a\right),\sigma^{n}\left(b\right)}\cap U_{c,d}$
is nonempty. This means that we need to show that for some $n<\omega$,
$\sigma^{n}\left(a\right)c\equiv\sigma^{n}\left(b\right)d$. As $\sigma$
is repulsive, there is some $n<\omega$ such that $cd\ind^{ns}\sigma^{n}\left(ab\right)$
and $\sigma^{n}\left(ab\right)\ind^{ns}cd$. Since both $U,V\neq\emptyset$,
$a\equiv b$ and $c\equiv d$, and thus $\sigma^{n}\left(a\right)c\equiv\sigma^{n}\left(b\right)c\equiv\sigma^{n}\left(b\right)d$. 
\end{proof}
\begin{fact}
\label{fact:TT implies DO}\cite[Proposition 1.1]{MR1159963} If $X$
is second countable (i.e., have a countable basis) of second category
(i.e., not meager) and separable, and $f:X\to X$ is topologically
transitive, then for some $x\in X$, $\set{f^{n}\left(x\right)}{n<\omega}$
is dense. 

In fact, it follows from the proof there that the set of such $x$'s
is comeager. 
\end{fact}

\begin{proof}
 Since it is not written explicitly in \cite{MR1159963}, we provide
a proof of the last statement (based on the proof from there). Consider
the set $F$ of $x\in X$ such that $\set{f^{n}\left(x\right)}{n<\omega}$
is not dense. Fix some countable basis $\cal V$ of open sets. For
each $x\in F$, there is some $U_{x}\in\cal V$ such that $f^{n}\left(x\right)\notin U_{x}$
for all $n$. Now, $\bigcup\set{f^{-n}\left(U_{x}\right)}{n<\omega}$
is open and dense since $f$ is topologically transitive. Hence, its
complement $A_{U_{x}}$ is closed, nowhere dense and contains $x$.
The union $\bigcup\set{A_{U_{x}}}{x\in F}$ is a countable union which
contains $F$, hence $F$ is meager. 
\end{proof}
In our case, $G=\Aut\left(M\right)$ for a countable model $M$ is
of second category, since it is even Polish. Hence we immediately
get the following corollary.
\begin{cor}
\label{cor:repulsive implies cyclically dense}Suppose that $M$ is
a countable structure. Suppose that $G=\Aut\left(M\right)$ contains
a repulsive automorphism $\sigma$. Then $G$ is topologically 2-generated
and moreover has a cyclically dense conjugacy class, and even: the
set of $\tau$ for which $\set{\sigma^{n}\tau\sigma^{-n}}{n\in\Nn}$
is dense is comeager. 
\end{cor}

An alternative, more direct proof is as follows. Assume that $\sigma$
is a repulsive automorphism, and construct $\tau\in\Aut\left(M\right)$
by back-and-forth, so that $\set{\sigma^{-n}\tau\sigma^{n}}{n\in\Zz}$
is dense in $G$.  We leave the details as an exercise.  We also
point out that according to \cite[Lemmas 5.1 and 5.2]{Kwiatkowska2016},
the fact that this set of $\tau$ is comeager actually follows immediately
from the fact that there is one such $\tau$. 

Now we turn to the question of finding a repulsive automorphism. 
\begin{defn}
\label{def:groovy}Suppose that $M$ is a countable structure. A ternary
relation $\ind$ on finite subsets of $M$, invariant under $\Aut\left(M\right)$
is a \emph{canonical} independence relation (CIR) if it satisfies
the following properties for all finite sets $A,B,C,D$:
\begin{itemize}
\item (Stationarity over $\emptyset$) If $A\ind B$, $A'\ind B'$, $a,b,a',b'$
are tuples enumerating $A,B,A',B'$, and $a\equiv a'$, $b\equiv b'$
then $ab\equiv a'b'$. Note that this (together with monotonicity,
see below) implies non-splitting: if $A\ind B$ then $A\ind^{ns}B$
and $B\ind^{ns}A$ (where $A\ind B$ means $A\ind_{\emptyset}B$).
\item (Extension (on the right)) If $A\ind_{C}B$ then for all finite tuples
$d$ there is some $d'\equiv_{BC}d$ such that $A\ind_{C}B\cup d'$. 
\item (Transitivity on both sides) If $A\ind_{DC}B$ then if $D\ind_{C}B$
then $AD\ind_{C}B$ and if $A\ind_{C}D$ then $A\ind_{C}BD$ ($DC$
means $D\cup C$). 
\item (Monotonicity) If $A\ind_{C}B$ and $A'\subseteq A$, $B'\subseteq B$
then $A'\ind_{C}B'$. 
\item (Existence) $A\ind_{C}C$ and $C\ind_{C}A$. 
\end{itemize}
For finite tuples $a,b,c$ enumerating sets $A,B,C$ respectively,
write $a\ind_{c}b$ for $A\ind_{C}B$. 

Note that we do not ask for symmetry nor for base monotonicity (if
$A\ind_{C}BD$ then $A\ind_{CD}B$). 

We say that a CIR is \emph{defined on finitely generated substructures}
if for all finite $A,B,C,A',B',C'\subseteq M$, if $\left\langle A\right\rangle =\left\langle A'\right\rangle $,
$\left\langle B\right\rangle =\left\langle B'\right\rangle $ and
$\left\langle C\right\rangle =\left\langle C'\right\rangle $ then
$A\ind_{C}B$ iff $A'\ind_{C'}B'$. ($\left\langle A\right\rangle $
is the substructure generated by $A$). 
\end{defn}

\begin{rem}
\label{rem:CIR on f.g. substructures} If a CIR is defined on finitely
generated substructures, then it naturally induces a relation $\ind^{*}$
whose domain is finitely generated substructures by setting $\left\langle A\right\rangle \ind_{\left\langle C\right\rangle }^{*}\left\langle B\right\rangle $
iff $A\ind_{C}B$. The relation $\ind^{*}$ satisfies the natural
variants of Definition \ref{def:groovy}. For example, transitivity
to the left becomes: for all finitely generated substructures $A,B,C,D$,
if $A\ind_{\left\langle DC\right\rangle }^{*}B$ and $D\ind_{C}^{*}B$
then $\left\langle AD\right\rangle \ind_{C}^{*}B$. Similarly, extension
becomes: if $A\ind_{C}^{*}B$ then for all finite tuples $d$ there
is some $d'\equiv_{BC}d$ such that $A\ind_{C}^{*}\left\langle Bd'\right\rangle $.

On the other hand, if we have a relation $\ind^{*}$ satisfying these
natural properties on finitely generated substructures of $M$, then
there is also a CIR defined on finitely generated substructures: define
$A\ind_{C}B$ iff $\left\langle A\right\rangle \ind_{\left\langle C\right\rangle }^{*}\left\langle B\right\rangle $.
\end{rem}

\begin{thm}
\label{thm:existence of repulsive automorphism-ultrahomogeneous}
Assume that $\ind$ is a CIR defined on finitely generated substructures
of an ultrahomogeneous structure $M$. Then there is a strongly repulsive
automorphism in $\Aut\left(M\right)$. 
\end{thm}

\begin{proof}
Let $S$ be the set of all closed nonempty intervals of integers,
i.e., sets of the form $\left[i,j\right]$ for $i\leq j$ from $\Zz$.
For every finite set $s\in S$ we attach a countable tuple of variables
$\bar{x}_{s}=\sequence{x_{s,i}}{i<\omega}$ in such a way that if
$t\neq s\in S$ then $\bar{x}_{s}\cap\bar{x}_{t}=\emptyset$. For
$s\in S$, let $\bar{y}_{s}=\bigcup\set{\bar{x}_{t}}{t\in S,t\subseteq s}$,
and $\bar{y}=\bigcup\set{\bar{x}_{s}}{s\in S}=\bigcup\set{\bar{y}_{s}}{s\in S}$.
For a $\bar{y}$-tuple ($\bar{y}_{s}$-tuple) $\bar{a}$ and $t\in S$
(contained in $s$), we write $\bar{a}\upharpoonright t$ for $\bar{a}\restriction\bar{y}_{t}$
and similarly for a type in $\bar{y}$ ($\bar{y}_{s}$). 

By Fact \ref{fact:Fraisse limits}, the age of $M$, denoted by $K$,
has HP, JEP and AP. Fix an enumeration $\sequence{\left(A_{l},B_{l}\right)}{l<\omega}$
of all pairs $A,B\in K$ such that $A\subseteq B\subseteq M$, including
the case $A=\emptyset$.

For every $1\leq n<\omega$, we construct a complete quantifier free
type $r_{n}\left(\bar{y}_{\left[0,n-1\right]}\right)\in S^{\qf}\left(\emptyset\right)$
and a sequence $\sequence{f_{n,l,i}}{l,i<\omega}$ such that:
\begin{enumerate}
\item \label{enu:model}If $\bar{a}\models r_{n}$ then $\bar{a}$ is enumerates
a finitely generated substructure $A\in K$ and for a fixed $l<\omega$,
$\sequence{f_{n,l,i}}{i<\omega}$ enumerates a countable set of functions
that contains all embeddings of $A_{l}$ into $A$. (Formally, $f_{n,l,i}$
is a function from $A_{l}$ into the variables $\bar{y}_{\left[0,n-1\right]}$
which the type $r_{n}$ ``thinks'' is an embedding.)
\item \label{enu:all intervals are the same}If $0<m<n$ then for every
interval $s\in S$ with $s\subseteq n$ such that $\left|s\right|=m$,
$r_{n}\restriction s=r_{m}\left(\bar{y}_{s}\right)$.
\item \label{enu:AP} If $\bar{a}\models r_{n}$ then for every $0<m<n$,
every pair from $\set{\left(A_{l},B_{l}\right)}{l<n-1}$ and every
$f\in\set{f_{m,l,i}}{l,i<n-1}$ (so $f$ is an embedding of $A_{l}$
into the structure enumerated by $\bar{a}\restriction\left[0,m-1\right]$),
 there is an embedding $g$ of $B_{l}$ into the structure enumerated
by $\bar{a}$ such that $g$ extends $f$. 
\item \label{enu:independence}If $s,t$ are two intervals contained in
$n$ such that $\min s\leq\min t$ and $\bar{a}\models r_{n}$ then
$\bar{a}\restriction s\ind_{\bar{a}\restriction s\cap t}^{*}\bar{a}\restriction t$
(see Remark \ref{rem:CIR on f.g. substructures}). 
\end{enumerate}
How? 

For $n=1$, let $r_{1}\left(\bar{y}_{\left\{ 0\right\} }\right)$
be a complete quantifier free type of a tuple enumerating some $D_{0}\in K$.
 Note that it trivially satisfies all the assumptions. 

Suppose we found $r_{n}$ satisfying all the properties and we construct
$r_{n+1}$. Let $\bar{a}\models r_{n}$ from $M$. Let $f:\bar{a}\restriction\left[0,n-2\right]\to\bar{a}\restriction\left[1,n-1\right]$
be defined by setting $f\left(a_{s,i}\right)=a_{s+1,i}$ for all $s\in S$
contained in $\left[0,n-2\right]$ (in general, $s+n$ is the translation
of $s$ by $n$). It is an isomorphism since both $\bar{a}\restriction\left[0,n-2\right]$
and $\bar{a}\restriction\left[1,n-1\right]$ realize $r_{n-1}$. By
the homogeneity of $M$, we can extend $\bar{a}\restriction\left[1,n-1\right]$
to some tuple $\bar{a}'$ enumerated by $\bar{y}_{\left[1,n\right]}$,
and extend $f$ to an isomorphism $f':\bar{a}\to\bar{a}'$ such that
$f'\left(a_{s,i}\right)=a_{s+1,i}$ for all $s\in S$ contained in
$\left[0,n-1\right]$. (When $n=1$, $f=\emptyset$ and $\bar{a}'$
is just an isomorphic copy of $\bar{a}$.)

By existence, we have that $\bar{a}\ind_{\bar{a}\restriction\left[1,n-1\right]}^{*}\bar{a}\restriction\left[1,n-1\right]$.
By extension, we may assume that $\bar{a}\ind_{\bar{a}\restriction\left[1,n-1\right]}^{*}\bar{a}'$.
Let $\bar{a}''$ be a $\bar{y}_{n+1}$-tuple containing $\bar{a}\bar{a}'$
that contains witnesses to all the relevant pairs and embeddings from
(\ref{enu:AP}), ordered in such a way that $\bar{a}$ enumerates
the $\bar{y}_{n}$-part and $\bar{a}'$ enumerates the $\bar{y}_{\left[1,n\right]}$-part
(so the remaining parts of $\bar{a}''$ are enumerated by $\bar{x}_{\left[0,n\right]}$).
This can be done since $K$ has AP and JEP. Let $r_{n+1}=\tp\left(\bar{a}''\right)$. 

Now we have to check that (1)\textendash (4) hold. We prove this by
induction on $n$. 

(\ref{enu:model}) and (\ref{enu:AP}) are clear by construction and
the induction hypothesis. (\ref{enu:all intervals are the same})
follows by the choice of $f$ and $f'$. 

Let us prove (\ref{enu:independence}), so fix $\bar{a}\models r_{n+1}$.
By monotonicity it is enough to prove that $\bar{a}\upharpoonright\left[0,k-1\right]\ind_{\bar{a}\restriction\left[m,k-1\right]}^{*}\bar{a}\restriction\left[m,n\right]$
for any $m\leq n,1\leq k\leq n+1$. We may assume that $1\leq m$
and $k\leq n$ (otherwise this is true by existence).

Note that $\bar{a}\restriction\left[0,k-1\right]\ind_{\bar{a}\restriction\left[1,n-1\right]}^{*}\bar{a}\restriction\left[m,n\right]$
by construction (and monotonicity). By induction, $\bar{a}\restriction\left[1,n-1\right]\ind_{\bar{a}\restriction\left[m,n-1\right]}^{*}\bar{a}\restriction\left[m,n\right]$.
Hence by transitivity and monotonicity we have that $\bar{a}\restriction\left[0,k-1\right]\ind_{\bar{a}\restriction\left[m,n-1\right]}^{*}\bar{a}\restriction\left[m,n\right]$.
By induction we have that $\bar{a}\upharpoonright\left[0,k-1\right]\ind_{\bar{a}\restriction\left[m,k-1\right]}^{*}\bar{a}\restriction\left[m,n-1\right]$.
By applying transitivity (and monotonicity) again, we get that $\bar{a}\upharpoonright\left[0,k-1\right]\ind_{\bar{a}\restriction\left[m,k-1\right]}^{*}\bar{a}\restriction\left[m,n\right]$.
This finishes the proof of (4). 

By compactness, we can find a $\bar{y}$-tuple $\bar{b}$, such that
for every $s\in S$, $\bar{b}\restriction s\models r_{\left|s\right|}$.
Note that $\bar{b}$ enumerates a Fra\"iss\'e limit $N$ by (\ref{enu:AP}).
This is as in the proof of \cite[Theorem 7.1.2]{Hod}. More precisely,
by JEP, the age of $N$ is $K$ (given $A\in K$, by JEP there is
some $B$ such that $\left(D_{0},B\right)$ is one of $\left(A_{l},B_{l}\right)$
and $A\subseteq B$, so this is taken care of in the construction).
In addition, by \cite[Lemma 7.1.4 (b)]{Hod}, $N$ is ultrahomogeneous.

Hence by uniqueness of the Fra\"iss\'e limit (see Fact \ref{fact:Fraisse limits}),
we may assume that $\bar{b}$ enumerates $M$. Let $\sigma:M\to M$
be defined by $\sigma\left(b_{s,i}\right)=b_{s+1,i}$. By construction,
$\sigma$ is an automorphism. Now, any finite subset of $M$ is contained
in $\bar{b}\restriction s$ for some $s\in S$. Then for some $m<\omega$,
$s+n\cap s=\emptyset$ for all $n\geq m$, so $\sigma^{n}\left(\bar{b}\upharpoonright s\right)=\bar{b}\restriction\left(s+n\right)$
satisfies that $\sigma^{n}\left(\bar{b}\upharpoonright s\right)\ind^{*}\bar{b}\upharpoonright s$
and so as $\ind^{*}$ implies non-splitting, $\sigma$ is indeed strongly
repulsive.
\end{proof}

\begin{cor}
\label{cor:CIR in omega-cat implies repulsive}Suppose that $M$ is
a countable $\omega$-categorical $L$-structure, and that $\ind$
is a CIR on $M$. Then there is a strongly repulsive automorphism
$\sigma\in\Aut\left(M\right)$. 
\end{cor}

\begin{proof}
For all $n<\omega$ and $a\in M^{n}$, let $R_{a}\subseteq M^{n}$
be the orbit of $a$ under $G=\Aut\left(M\right)$. Let $L'=\set{R_{a}}{a\in M^{n},n<\omega}$
and let $M'$ be the $L'$-structure induced by $M$. Then $M'$ has
the same definable sets as $M$, has quantifier elimination and is
ultrahomogeneous. Now, $\ind$ is still a CIR on finite subsets of
$M'$. Moreover, substructures of $M'$ are subsets since $L'$ is
relational, so $\ind$ is defined on finitely generated substructures.
Thus we may apply Theorem \ref{thm:existence of repulsive automorphism-ultrahomogeneous}
to get a strongly repulsive automorphism $\sigma$ of $M'$. However,
$\Aut\left(M\right)=\Aut\left(M'\right)$ and $\sigma$ is strongly
repulsive as an automorphism of $M'$ as well. 
\end{proof}
From Corollary \ref{cor:repulsive implies cyclically dense} and Corollary
\ref{cor:CIR in omega-cat implies repulsive} we get:
\begin{cor}
\label{cor:having a groovy implies dense}If $M$ is a countable model
of an $\omega$-categorical theory which has a canonical independence
relation then $G=\Aut\left(M\right)$ has a cyclically dense conjugacy
class. In fact, there is some $f\in G$ such that the set of $g\in G$
for which {[}$\set{f^{n}gf^{-n}}{n\in\Zz}$ is dense{]} is comeager. 

Furthermore, by Proposition \ref{prop:product}, the same is immediately
true for $G^{n}$ for any $n<\omega$. 
\end{cor}

\subsection{Ramsey structures and a weakening of having a CIR}
\begin{lem}
\label{lem:CIR implies just over 0}Suppose that $M$ is a countable
ultrahomogeneous structure. Then (1) implies (2) implies (3) where:
\begin{enumerate}
\item $M$ has a CIR defined on finitely generated substructures. 
\item There are two models $M_{0},M_{1}$ isomorphic to $M$ and contained
in $M$, such that $M_{0}\ind^{ns}M_{1}$ and $M_{1}\ind^{ns}M_{0}$.
\item There is a binary relation $\ind$ on finite subsets of $M$ that
satisfies all the properties of Definition \ref{def:groovy} but only
over $\emptyset$. Namely it satisfies stationarity, extension to
the right and the left (over $\emptyset$), monotonicity (over $\emptyset$)
and existence (over $\emptyset$).
\end{enumerate}
\end{lem}

\begin{proof}
(1) implies (2). Suppose that $M$ has a CIR $\ind$ defined on finitely
generated substructures. Consider the tuple $\bar{b}$ constructed
in the proof of Theorem \ref{thm:existence of repulsive automorphism-ultrahomogeneous}.
By condition (\ref{enu:AP}) in that proof, both $\bar{b}\restriction\left[0,\infty\right)$
and $\bar{b}\restriction\left(-\infty,0\right]$ are ultrahomogeneous
and with the same age as $M$, thus isomorphic to $M$ by Fact \ref{fact:Fraisse limits}.
By stationarity and monotonicity, $\bar{b}\restriction\left[0,\infty\right)\ind^{ns}\bar{b}\restriction\left(-\infty,0\right]$
and $\bar{b}\restriction\left(-\infty,0\right]\ind^{ns}\bar{b}\restriction\left[0,\infty\right)$. 

(2) implies (3). For two finite sets $A,B\subseteq M$, let $A\ind B$
iff there is some automorphism $\sigma\in\Aut\left(M\right)$ such
that $\sigma\left(A\right)\subseteq M_{0}$ and $\sigma\left(B\right)\subseteq M_{1}$.
Now, $\ind$ is stationary: suppose that $a\ind b$, $a'\ind b'$,
$a\equiv b$ and $a'\equiv b'$. We want to show that $ab\equiv a'b'$.
Let $\sigma,\tau\in\Aut\left(M\right)$ be such that $\sigma\left(a\right),\tau\left(a'\right)\subseteq M_{0}$
and $\sigma\left(b\right),\tau\left(b'\right)\subseteq M_{1}$. Then
$\sigma\left(a\right)\tau\left(a'\right)\ind^{ns}\sigma\left(b\right)\tau\left(b'\right)$
and $\sigma\left(b\right)\tau\left(b'\right)\ind^{ns}\sigma\left(a\right)\tau\left(a'\right)$,
and hence $ab\equiv\sigma\left(ab\right)\equiv\sigma\left(a\right)\tau\left(b'\right)\equiv\tau\left(a'b'\right)\equiv a'b'$. 

Next, $\ind$ satisfies extension (on the right): suppose that $A\ind B$
and $d$ is a finite tuple of $M$, and we want to find some $d'\equiv_{B}d$
such that $A\ind B\cup d'$. By definition, there is some $\sigma$
with $\sigma\left(A\right)\subseteq M_{0}$ and $\sigma\left(B\right)\subseteq M_{1}$.
As $M_{1}$ is ultrahomogeneous and $B$ is finite, we can extend
$\sigma\restriction\left\langle B\right\rangle $ to some $f:\left\langle Bd\right\rangle \to M_{1}$.
As $M$ is ultrahomogeneous, $d'=\sigma^{-1}\left(f\left(d\right)\right)\equiv_{B}d$
and $\sigma$ witnesses that $A\ind B\cup d'$. Extension on the
left is shown in the same way. 

Existence follows from the fact that $\Age\left(M_{1}\right)=\Age\left(M_{0}\right)=\Age\left(M\right)$.
Monotonicity is clear.

\end{proof}
\begin{rem}
If in Lemma \ref{lem:CIR implies just over 0}, if we had assumed
that $M$ was $\omega$-categorical, in (3) we could define $\ind$
on arbitrary subsets of $M$, even infinite. We would define $A\ind B$
iff for every finite subsets $A'\subseteq A$ and $B'\subseteq B$,
$A'\ind B'$. One can then use K\"onig's Lemma and Ryll-Nardzewski
to show that this has the extension property.  

In addition, (3) would be equivalent to (2): enumerate $M$ as $\bar{m}=\sequence m{m\in M}$
(i.e., the identity function). Let $\bar{x}=\sequence{x_{m}}{m\in M}$
and $\bar{y}=\sequence{y_{m}}{m\in M}$ be two disjoint sequences
of variables. For a finite tuple $a=\sequence{m_{i}}{i<n}$ in $M$,
write $x_{a}=\sequence{x_{m_{i}}}{i<n}$, and similarly define $y_{a}$.
Let $\Gamma\left(\bar{x},\bar{y}\right)$ be the union of the sets
$\Gamma_{a,b,c}^{0}\left(x_{a},x_{b},y_{c}\right)$ and $\Gamma_{a,b,c}^{1}\left(x_{c},y_{a},y_{b}\right)$
for all finite tuples $a,b,c$ from $M$ such that $a\equiv b$, where
$\Gamma_{a,b,c}^{0}\left(x_{a},x_{b},y_{c}\right)$ says that $x_{a}$
and $x_{b}$ have the same type over $y_{c}$ and similarly, $\Gamma_{a,b,c}^{1}\left(x_{c},y_{a},y_{b}\right)$
says that $y_{a}$ and $y_{b}$ have the same type over $x_{c}$.
Let $\Sigma\left(\bar{x},\bar{y}\right)$ be $\Gamma\left(\bar{x},\bar{y}\right)$
and the assertions that both $\bar{x}$ and $\bar{y}$ satisfy the
type $\tp\left(\bar{m}/\emptyset\right)$. Then by (3), $\Sigma$
is consistent, so by $\omega$-categoricity, we can realize it in
$M$. 
\end{rem}

\begin{prop}
\label{prop:Ramsey has everything but over base}Suppose that $M$
is a countable $\omega$-categorical Ramsey structure. Then (2) from
Lemma \ref{lem:CIR implies just over 0} holds.
\end{prop}

\begin{proof}
Let $G=\Aut\left(M\right)$ and let $\bar{m}=\sequence m{m\in M}$
(i.e., the identity function). Let $\bar{x}=\sequence{x_{m}}{m\in M}$
and $S_{\bar{m}}\left(M\right)=\set{p\left(\bar{x}\right)\in S\left(M\right)}{p\restriction\emptyset=\tp\left(\bar{m}/\emptyset\right)}$,
a compact Hausdorff space with the logic topology. Then $G$ acts
on $S_{\bar{m}}\left(M\right)$ by setting $\sigma*p=\set{\sigma*\varphi}{\varphi\in p}$
where $\sigma*\varphi\left(\bar{x},m\right)=\varphi\left(\bar{x},\sigma\left(m\right)\right)$.
 A fixed point of this action is just an invariant type over $M$.
As $G$ is extremely amenable (Fact \ref{fact:ramsey iff extremeley amenable}),
there is an invariant type which enumerates a model $N$ such that
that $N\ind^{ns}M$.  Now consider the space $P$ of invariant types:
$P=\set{q\in S_{\bar{m}}\left(M\right)}{G*q=q}$. A type $q$ is invariant
iff for every formula $\varphi\left(\bar{x},y\right)$ over $\emptyset$,
if $m\equiv m'$ are from $M$ then $\varphi\left(\bar{x},m\right)\in q$
iff $\varphi\left(\bar{x},m'\right)\in q$. This is easily a closed
condition, so $P$ is compact, and we already know that it is nonempty.
Now let $G$ act on $P$ by setting $\sigma\star q=\set{\sigma\star\varphi}{\varphi\in q}$
where $\sigma\star\varphi\left(\bar{x},m\right)=\varphi\left(\bar{x}_{\sigma},m\right)$,
where $\bar{x}_{\sigma}=\sequence{x_{\sigma\left(m\right)}}{m\in M}$.
Note that for all $p\in P$ and $\sigma\in G$, $\sigma\star p\restriction\emptyset=\tp\left(\sequence{\sigma^{-1}\left(m\right)}{m\in M}/\emptyset\right)=\tp\left(\bar{m}/\emptyset\right)$,
and that $\sigma\star p$ remains invariant. By extreme amenability,
there is some $q\in P$ such that $\sigma\star q=q$ for all $\sigma\in G$.
Let $N'$ be a model enumerated by a realization of $q$, then $N'\ind^{ns}M$
and $M\ind^{ns}N'$.  

By $\omega$-categoricity, we may assume that these two models are
contained in $M$. 
\end{proof}
\begin{rem}
\label{rem:Ramsey need not have CIR} By Fact \ref{fact:Expanding Ramsey by naming elements},
expanding an ultrahomogeneous Ramsey structure by finitely many constants
gives an independence relation as in (3) from Lemma \ref{lem:CIR implies just over 0}
over any finite set. However there is no reason that it would satisfy
transitivity. Indeed, the lexicographically ordered dense tree (which
is Ramsey, see Example \ref{exa:lex trees are ramsey}) does not have
a CIR. See Corollary \ref{cor:Dense trees don't have a groovy ind relation}
below. 
\end{rem}

\section{Examples of theories with a canonical independence relation}

There are many examples of countable ultrahomogeneous structures with
a CIR. Here we will give some of them. We will define the relation
$\ind$, but sometimes leave most of the details of checking that
it satisfies the axioms to the reader. All the CIRs we define are
defined on finitely generated substructures. 
\begin{example}
\label{exa:trivial structure}The most trivial ultrahomogeneous structure
is of course the structure with universe $\omega$ and no relations
but equality. Its automorphism group is $S_{\infty}$. For finite
sets $A,B,C$ define $A\ind_{C}B$ by $A\cap B\subseteq C$. This
is a CIR. 
\end{example}

\begin{example}
\label{exa:T stable with emptyset stationary}If $T$ is stable and
$\emptyset$ is a base (i.e., $\acl^{\eq}\left(\emptyset\right)=\dcl^{\eq}\left(\emptyset\right)$
so that every type over $\emptyset$ has a unique non-forking extension),
then $\ind^{f}$ (i.e., non-forking independence) is canonical.
\end{example}

\begin{example}
Let $\left(\cal B,<,\wedge,\vee,0,1,\square^{c}\right)$ be the atomless
Boolean algebra. For finite sets $A,B,C$, define $A\ind_{C}B$ iff
$\left\langle C\right\rangle =\left\langle AC\right\rangle \cap\left\langle BC\right\rangle $
and for every atom $a\in\left\langle AC\right\rangle $ and every
atom $b\in\left\langle BC\right\rangle $, if there is an atom $c\in\left\langle C\right\rangle $
such that $a,b\leq c$ then $a\wedge b\neq0$. Let us show that $\ind$
is a (symmetric) CIR. 

Stationarity over $\emptyset$: $A\ind B$ says that the atoms in
$\left\langle AB\right\rangle $ are in bijection with (atoms of $A$)
$\times$ (atoms of $B$).  Thus, if $A'\ind B'$ and $f:\left\langle A\right\rangle \to\left\langle A'\right\rangle $,
$g:\left\langle B\right\rangle \to\left\langle B'\right\rangle $
are isomorphisms, then $h:\left\langle AB\right\rangle \to\left\langle A'B'\right\rangle $
taking an atom $a\wedge b$ to $f\left(a\right)\wedge g\left(b\right)$
is an isomorphism. This easily implies stationarity. 

Transitivity: suppose that $A\ind_{CD}B$ and $D\ind_{C}B$, and we
want to show that $AD\ind_{C}B$. Suppose that $a\in\left\langle ACD\right\rangle $,
$b\in\left\langle BC\right\rangle $ are atoms, and $c\in\left\langle C\right\rangle $
is an atom such that $a,b\leq c$. Let $d\in\left\langle DC\right\rangle $
be an atom such that $a\leq d\leq c$. Then $d\wedge b\neq0$. Let
$b'\leq d\wedge b$ be an atom of $\left\langle BCD\right\rangle $,
so that $a,b'\leq d$. Hence $a\wedge b'\neq0$ and thus $a\wedge b\neq0$.
Transitivity to the right follows by symmetry. 

Extension: suppose that $A\ind_{C}B$ and $d$ is given. We want to
find $d'\equiv_{BC}d$ such that $A\ind_{C}Bd$. We may assume that
$d\notin\left\langle BC\right\rangle $. An atom in $\left\langle BCd\right\rangle $
has the form $d\wedge b$ or $d^{c}\wedge b$ for some atom $b$ of
$\left\langle BC\right\rangle $. The type $\tp\left(d/BC\right)$
is determined by knowing which of these terms $d\wedge b,d^{c}\wedge b$
is nonzero (for $b\in\left\langle BC\right\rangle $ an atom). As
$\cal B$ is atomless and $A\ind_{C}B$, we can find some $d'$ such
that for every $a,b,c$ atoms in $\left\langle AC\right\rangle ,\left\langle BC\right\rangle ,\left\langle C\right\rangle $
respectively such that $a,b\leq c$, if $d\wedge b\neq0$ then $a\wedge d'\wedge b\neq0$
(else $d'\wedge b=0$) and if $d^{c}\wedge b\neq0$ then $a\wedge\left(d'\right)^{c}\wedge b\neq0$
(else $\left(d'\right)^{c}\wedge b=0$). In addition, we ask that
if $d\wedge b,d^{c}\wedge b\neq0$ then both $a\wedge d'\wedge b$
and $a\wedge\left(d'\right)^{c}\wedge b$ are not in $\left\langle ABC\right\rangle $.
It now follows that $\left\langle AC\right\rangle \cap\left\langle BCd'\right\rangle =\left\langle C\right\rangle $:
suppose that $e\in\left\langle AC\right\rangle \cap\left\langle BCd'\right\rangle $.
Then as $e\in\left\langle BCd'\right\rangle $, it can be written
as $b_{0}\vee\left(b_{1}\wedge d'\right)\vee\left(b_{2}\wedge\left(d'\right)^{c}\right)$
where $b_{0},b_{1},b_{2}\in\left\langle BC\right\rangle $ are pairwise
disjoint and for every atom $b'\leq b_{1}\vee b_{2}$ from $\left\langle BC\right\rangle $,
$d'\wedge b',\left(d'\right)^{c}\wedge b'\neq0$. If both $b_{1},b_{2}=0$,
then $e\in\left\langle BC\right\rangle \cap\left\langle AC\right\rangle $
so $e\in\left\langle C\right\rangle $. If $b_{1}\neq0$, let $b_{1}'\leq b_{1}$
be an atom of $\left\langle BC\right\rangle $. So $e\wedge b_{1}'\in\left\langle ABC\right\rangle $
(because $e\in\left\langle AC\right\rangle $) and has the form $b_{1}'\wedge d'$.
Let $a\in\left\langle AC\right\rangle $ and $c\in\left\langle C\right\rangle $
be atoms such that $a,b_{1}'\leq c$. Then $e\wedge b_{1}'\wedge a\in\left\langle ABC\right\rangle $
and has the form $a\wedge b_{1}'\wedge d'$ which is not in $\left\langle ABC\right\rangle $
by construction, contradiction. Similarly $b_{2}=0$ and we are done. 

Existence and monotonicity are clear. 
\end{example}

\begin{example}
Let $\left(M,R\right)$ be the random tournament (a tournament is
a complete directed graph such that for all $x,y$, it cannot be that
both $R\left(x,y\right)$ and $R\left(y,x\right)$, and the random
tournament is the Fra\"iss\'e limit of the class of finite tournaments).
Given finite sets $A,B,C$, write $A\ind_{C}B$ iff $A\cap B=C$ and
if $a\in A\backslash C$, $b\in B\backslash C$ then $R\left(a,b\right)$.
This easily satisfies all the requirements. 
\end{example}

\begin{rem}
The following definition is from \cite[after Theorem 5.13]{Kwiatkowska2016}.
Let $K$ be a class of finite $L$-structures where $L$ is a relational
language. Then $K$ has the s\emph{trong$^{+}$ amalgamation property}
if for all $A,B,C\in K$ with $A\subseteq B,C$ there is $D\in K$
such that $D=C'\cup B'$ with $C'\cong_{A}C$, $B'\cong_{A}B$ and
for every $n$-ary relation $R$ and every $x_{1},\ldots,x_{n}$ and
$y_{1},\ldots,y_{n}$ from $D\backslash A$ which intersect both $C'$
and $B'$, if {[}$x_{i}\in B'$ iff $y_{i}\in B'$ for all $1\leq i\leq n${]},
then $R^{D}\left(x_{1},\ldots,x_{n}\right)$ iff $R^{D}\left(y_{1},\ldots,y_{n}\right)$
(where $R^{D}$ is the interpretation of $R$ in $D$). For example,
the random tournament satisfies this property. In \cite[after Theorem 5.14]{Kwiatkowska2016}
it is proved that if $M$ is ultrahomogeneous and $\Age\left(M\right)$
has the strong$^{+}$ amalgamation property, then $\Aut\left(M\right)$
has a cyclically dense conjugacy class. They prove it using a condition
they denote by $(\Delta_{n})$, see there, Theorem 5.12. We do not
know if this condition implies the existence of a CIR.
\end{rem}

\subsection{Free amalgamation classes}

In \cite{MR3022717,Isabel,MR3663421} there is an axiomatic framework
for defining an abstract ternary relation close to our CIR. More precisely,
in \cite[Definition 3.1]{Isabel} and \cite[Definition 2.1]{MR3022717},
the notion of a stationary independence relation (SIR) is introduced
(in \cite[Definition 3.1]{Isabel} for finitely generated structures
and in \cite[Definition 2.1]{MR3022717} for sets in general). A similar
notion is defined in \cite[Definition 2.1]{MR3663421}, with more
axioms. In any case, all these notions imply ours, except perhaps
that stationarity over $\emptyset$ becomes stationarity over $\acl\left(\emptyset\right)$,
so that this becomes a CIR in the expansion $\left(M,\acl\left(\emptyset\right)\right)$
(note that our extension follows from full stationarity and their
version of existence). 

Thus, we can apply our results to the examples studied there. In particular,
we get the following examples.
\begin{example}
The rational Urysohn space $\Qq\Uu$ is the Fra\"iss\'e limit of
the class of finite metric spaces with rational distances. Pick a
point $q\in\Qq\Uu$, and consider the structure $\left(\Qq\Uu,q\right)$
where we add a constant for $q$. In \cite{TentZiegler}, it is proved
that the relation $A\ind_{C}B$ which holds for finite $A,B,C$ iff
for every $a\in AC$, $b\in BC$, $d\left(a,b\right)=\min\set{d\left(a,c\right)+d\left(c,b\right)}{c\in C}$
is a CIR in $\left(\Qq\Uu,q\right)$. 
\end{example}

In all examples given by Conant \cite[Example 3.2]{MR3663421} which
we list now, $\acl\left(\emptyset\right)=\emptyset$, so we actually
get a CIR in the structure (i.e., no need to take an expansion) by
\cite[Proposition 3.4]{MR3663421}. 
\begin{example}
\label{exa:free amalgamation classes}Fra\"iss\'e limits with free
amalgamation: suppose that $L$ is a relational language and $K$
is an essentially countable (see above Fact \ref{fact:Fraisse limits})
class of finite $L$-structures, such that if $A,B,C\in K$ and $A\subseteq C,B$,
then the free amalgam of $A,B,C$ is in $K$ (i.e., a structure $D=C'\cup B'$
with $C'\cong_{A}C$, $B'\cong_{A}B$, $B'\cap C'\subseteq A$ and
for every tuple $a\in D$ in the length of some relation $R\in L$,
if $R\left(a\right)$ then $a\in C'$ or $a\in B'$) (here we also
include the case $A=\emptyset$). Let $M$ be the Fra\"iss\'e limit
of $K$, and define $B\ind_{A}C$ iff $ABC$ is the free amalgam of
$A,AB,AC$. If the language is finite or more generally in the context
of Fact \ref{fact:Fraisse limits} (2), it is easy to see that in
this case $\ind$ is a CIR (this is also proved, for finite languages,
in \cite[Proposition 3.4]{MR3663421}).

This class of examples contain e.g., the random graph, the universal
$K_{n}$-free graph (the Henson graph), and their hypergraph analogs.
\end{example}

\begin{example}
\label{exa:naming the equivalence classes has SAP} Let $L=\set{P_{n}}{n<\omega}$
and let $K$ be the class of finite $L$-structures in which $P_{n}\left(x_{0},\ldots,x_{n-1}\right)$
implies that $x_{i}\neq x_{j}$ for $i\neq j$. Then $K$ is essentially
countable and is a free amalgamation class, and moreover for each
$n<\omega$ there are finitely many isomorphism types of structures
of size $n$ (so we can use Fact \ref{fact:Fraisse limits}). Let
$M$ be the limit. Now recall the example $N$ described in Remark
\ref{rem:constraint for cdcc compact quotient}, with infinitely many
independent equivalence relations with two classes. Then $M$ is $N$
expanded by naming the classes.  In other words, $\Aut\left(M\right)=\Aut\left(N\right)^{0}$. 
\end{example}

\begin{example}
In \cite[Section 10]{MR1683298} the authors describe a generic $K_{n}+K_{3}$-free
graph, where $K_{n}+K_{3}$ is the free amalgam of the complete graph
on $n$ vertices and a triangle over a single vertex. This structure
is $\aleph_{0}$-categorical, with $\acl\left(\emptyset\right)=\emptyset$.
By \cite[Example 3.2 (2), Proposition 3.4]{MR3663421} there is a
CIR on this graph. 
\end{example}

\begin{example}
$\omega$-categorical Hrushovski constructions. Let $L$ be a finite
relational language, and let $f:\Rr_{\geq0}\to\Rr_{\geq0}$ be a ``control
function''. According to \cite[Theorem 3.5]{MR1900903}, there is
an $\omega$-categorical generic Hrushovski construction $M_{f}$
for a ``free amalgamation class'' $K_{f}$ if $f$ satisfies certain
conditions. It follows from \cite[Example 3.2 (3), Proposition 3.4]{MR3663421}
that given extra conditions on the algebraic closure, $M_{f}$ admits
a CIR. See more details in \cite{MR3663421,MR1900903}. 
\end{example}

\subsection{Ultrahomogenous partial orders}

In \cite{Glab2017} the authors prove that the automorphism group
of every ultrahomogeneous poset (partially order set) is topologically
2-generated. They also characterize when they have a cyclically dense
conjugacy class. We can find such a conjugacy class by finding a CIR
whenever possible. We should remark that they prove more on the automorphism
groups of those structures. By \cite{MR544855} there are four types
of ultrahomogeneous posets. 
\begin{fact}
\cite{MR544855} Suppose that $\left(H,<\right)$ is an ultrahomogeneous
poset. Then $H$ is isomorphic to one of the following:
\begin{enumerate}
\item The random poset: the Fra\"iss\'e limit of the class of finite partial
orders. 
\item The orders $\cal A_{n}$ for $1\leq n\leq\omega$: $\left(n,<\right)$
where $<$ is trivial i.e., empty. 
\item The orders $\cal B_{n}$ for $1\leq n\leq\omega$: $\left(n\times\Qq,<\right)$
where $\left(k,q\right)<\left(m,p\right)$ iff $k=m$ and $p<q$. 
\item The orders $\cal C_{n}$ for $1\leq n\leq\omega$: $\left(n\times\Qq,<\right)$
where $\left(k,q\right)<\left(m,p\right)$ iff $q<p$.
\end{enumerate}
\end{fact}

Note that the orders $\cal A_{n}$ have $S_{n}$ as their automorphism
group, and thus for $n$ finite cannot have a dense conjugacy class.
For $n=\omega$, this is Example \ref{exa:trivial structure}. 

Also, the orders $\cal B_{n}$ for $1<n<\omega$ cannot have a dense
conjugacy class by Remark \ref{rem:constraint for cdcc compact quotient}:
$S_{n}$ is a quotient of the automorphism group (define $a\mathrela Eb$
iff $a$ and $b$ are comparable, and note that there are $n$ equivalence
classes, every permutation of which is induced by an automorphism). 

\subsubsection{The random poset}

Suppose that $\left(\cal D,\leq\right)$ is the random partial order.
For finite sets $A,B,C$ define $A\ind_{C}B$ iff $A\cap B\subseteq C$
and if $a\in A,b\in B$ then $a$ is comparable with $b$ iff for
some $c\in C$, $a\leq c\leq b$ or $b\leq c\leq a$. Then $\ind$
is a (symmetric) CIR. We will show only transitivity and extension,
and leave the rest to the reader. Suppose that $A\ind_{CD}B$ and
$D\ind_{C}B$. Given $a\in A,b\in B$, such that $a\leq b$, there
must be some $d\in CD$ such that $a\leq d\leq b$. Hence there must
be some $c\in C$ with $d\leq c\leq b$. Together we are done. Extension:
suppose that $A\ind_{C}B$ and we are given $d$. Assume $d\notin BC$
(otherwise we are done). Then let $d'\equiv_{BC}d$ be such that $d'\notin ABC$
and for all $a\in A$, $a\leq d'$ iff for some $c\in C$, $a\leq c\leq d$,
and similarly define when $d'\leq a$. Now, $\left(ABCd',\leq\right)$
is a poset since $A\ind_{C}B$. 

\subsubsection{The orders $\protect\cal B_{1}$ and $\protect\cal B_{\omega}$. }
\begin{example}
\label{exa:DLO has groovy}For $\left(\Qq,<\right)$ (which is $\cal B_{1}$),
for every finite $A,B,C\subseteq\Qq$, we let $A\ind_{C}B$ if $A\cap B\subseteq C$
and for all $a\in A\backslash C$ and $b\in B\backslash C$ such that
$a\equiv_{C}b$, $a<b$. Then $\ind$ is a CIR. We prove only transitivity
and leave the rest to the reader.

Suppose that $A\ind_{DC}B$ and $D\ind_{C}B$. We have to show that
$AD\ind_{C}B$, which amounts to showing that $A\ind_{C}B$. Fix $a\in A\backslash C$
and $b\in B\backslash C$ such that $a\equiv_{C}b$. We have to show
that $a<b$. Note that $b\notin D$. If $a\in D$ then this is true
by our assumption. Otherwise, $a\notin CD$. If $a\equiv_{CD}b$ then
we are done. Otherwise, $a,b$ have different cuts over $CD$. But
since they realize the same cut over $C$, it follows that there is
some $d\in D\backslash C$ such that either $a<d<b$ or $b<d<a$.
The former would imply what we want, so assume that $b<d<a$. But
then $b\equiv_{C}d$ so $d<b$ \textemdash{} a contradiction. The
other direction of transitivity is proved similarly. 
\end{example}

\begin{example}
Consider $\cal B_{\omega}$. Then each equivalence class of the relation
$E$ of being comparable is a DLO, and thus by Example \ref{exa:DLO has groovy},
for each $n<\omega$, there is a CIR $\ind^{F}$ defined as in Example
\ref{exa:DLO has groovy} for each $E$-class $F$. For $\left(\cal B_{\omega},<\right)$
and finite sets $A,B,C$, define $A\ind_{C}B$ iff $A\cap B\subseteq C$,
$A/E\cap B/E\subseteq C/E$(if $a\in A$, $b\in B$ and $a\mathrela Eb$
then there is some $c\in C$ such that $a\mathrela Ec$) and for every
$E$-class $F$, $A\cap F\ind_{C\cap F}^{F}B\cap F$. This is easily
seen to be a CIR. Note that we need infinitely many classes for extension.
\end{example}

\subsubsection{The orders $\protect\cal C_{n}$ for $1\protect\leq n\protect\leq\omega$. }

In $\cal C_{n}$ we have an equivalence relation $E$, defined by
$a\mathrela Eb$ iff $a$ and $b$ are incomparable (they have the
same second coordinate). Then $\cal C_{n}/E\models DLO$, so we have
a CIR $\ind^{E}$ defined on it by Example \ref{exa:DLO has groovy}.
For finite $A,B,C\subseteq\cal C_{n}$, define $A\ind_{C}B$ iff $A/E\ind_{C/E}^{E}B/E$.
This trivially satisfies all the properties. 

\subsection{Ultrahomogeneous graphs}

In \cite{MR3669009}, the authors prove that for every ultrahomogeneous
graph $\Gamma=\left(V,E\right)$, $\Aut\left(\Gamma\right)$ is topologically
2-generated. Similarly to the poset case, we can recover some results
by finding a CIR whenever possible. By \cite{MR583847} we have the
following classification of ultrahomogeneous graphs. Recall that for
a graph $\left(V,E\right)$, its \emph{dual} is $\left(V,E'\right)$
where $E'=\left[V\right]^{2}\backslash E$.
\begin{fact}
\cite{MR583847} Any countable ultrahomogeneous graph $\Gamma$ is
isomorphic to one of the following graphs, or its dual.
\begin{enumerate}
\item The random graph.
\item For $n\geq3$, the Henson graph, i.e., the $K_{n}$-free universal
graph (the Fra\"iss\'e limit of the class of $K_{n}$-free finite
graphs). 
\item For any $1\leq n\leq\omega$, the graph $\omega K_{n}$ consisting
of a disjoint union of countably many copies of $K_{n}$. 
\item For any $2\leq n<\omega$, the graph $nK_{\omega}$ consisting of
a disjoint union of $n$ copies of $K_{\omega}$ (the complete graph
on $\omega$). 
\end{enumerate}
\end{fact}

Note that the dual of a graph has the same automorphism group, so
we can ignore the duals. 

We already saw in Example \ref{exa:free amalgamation classes} that
both the random graph and the Henson graph have a CIR. The graphs
$nK_{\omega}$ for $n<\omega$ cannot have a a dense conjugacy class
by Remark \ref{rem:constraint for cdcc compact quotient} as in the
case of the posets $\cal B_{n}$ described above. However, $\omega K_{n}$
for $1\leq n\leq\omega$ has a CIR, just like the cases $\cal C_{n}$
above. 

\subsection{\label{subsec:A-mix-of-CIRs}A mix of two Fra\"is\'e limits with
CIRs}

Suppose we are in the situation of Section \ref{subsec:A-mix-of}:
we have two amalgamation classes $K_{1},K_{2}$ with all the properties
listed there. Let $M_{1},M_{2}$ be the Fra\"is\'e limits of $K_{1},K_{2}$
respectively, and let $M$ be the Fra\"is\'e limit of $K$, the
class of finite $L_{1}\cup L_{2}$-structures $A$ such that $A\restriction L_{1}\in K_{1}$
and $A\restriction L_{2}\in K_{2}$. Add the extra assumption that
$L_{1}\cap L_{2}=\emptyset$. Suppose that $\ind^{1},\ind^{2}$ are
CIRs on $M_{1},M_{2}$ respectively. By Proposition \ref{prop:restriction of a mix of two Faisse},
we may assume that $M_{1}=M\restriction L_{1}$ and $M_{2}=M\restriction L_{2}$.
For finite subsets of $M$, define $A\ind_{C}B$ iff $A\ind_{C}^{1}B$
and $A\ind_{C}^{2}B$.
\begin{prop}
\label{prop:A mix of CIRs is a CIR}The relation $\ind$ is a CIR. 
\end{prop}

\begin{proof}
Stationarity follows from the fact that by quantifier elimination,
for any finite tuples $a,a'$ from $M$, if $a\equiv a'$ in $L_{1}$
and in $L_{2}$, then $a\equiv a'$ in $L_{1}\cup L_{2}$. 

Extension: suppose that $A\ind_{C}B$, and we are given $d\in M$.
Let $d_{1}\in M_{1}$ be such that $d_{1}\equiv_{BC}d$ in $L_{1}$
and $A\ind_{C}^{1}Bd_{1}$. Similarly find $d_{2}$ for $\ind^{2}$
and $L_{2}$. Consider the finite structure $D$ with universe $ABCd$
where the $L_{1}\cup L_{2}$-structure on $ABC$ is as in $M$, and
such that its restriction to $L_{1}$, $L_{2}$ is $ABCd_{1}$, $ABCd_{2}$,
respectively. This structure exists since $L_{1}\cap L_{2}=\emptyset$
and by the assumptions of Section \ref{subsec:A-mix-of}, both languages
are relational. Thus, $D\in K$, so it has an isomorphic copy $D'\subseteq M$
containing copies $A',B',C',d'$ of $A,B,C,d$. As $M$ is ultrahomogeneous,
we can apply an automorphism $\sigma$ mapping $A'B'C'$ to $ABC$,
so that $A\ind_{C}B\sigma\left(d'\right)$, and $\sigma\left(d'\right)\equiv_{BC}d$.

The other properties are easy to check. 
\end{proof}
\begin{example}
\label{exa:The-random-ordered}The ordered random graph $M=\left(V,R,<\right)$.
It is the Fra\"is\'e limit of the class of finite linearly ordered
graphs in the language $\left\{ <,R\right\} $. It easily satisfies
all our assumptions with $L_{1}=\left\{ <\right\} $, $K_{1}$ the
class of finite linear orders and $L_{2}=\left\{ R\right\} $, $K_{2}$
the class of finite graphs. It has a CIR as both $\left(M,<\right)$
(which is a DLO) and $\left(M,R\right)$ (the random graph) have CIRs
by the two previous subsections. Similarly we may define the random
ordered hypergraph, and it too has a CIR. 
\end{example}

\subsection{\label{subsec:Trees}Trees}

The theory $T_{dt}$ (see Example \ref{exa:dense trees}) does not
admit a canonical independence relation. We shall give a precise (and
stronger) argument for this below in Corollary \ref{cor:Dense trees don't have a groovy ind relation},
but it is easy to see that a natural candidate fails. Namely, one
can try to define $A\ind_{C}B$ in such a way that if $C=\emptyset$
and $a,b$ are singletons then $a\ind b$ iff $a\wedge b<a,b$, and
for $a,b,c$ such that $c\ind b$, then $a\ind_{c}b$ iff $a\wedge c<c\wedge b$.
But then $a\wedge c\ind_{c}b$, $c\ind b$ but $a\wedge c\nind b$,
so transitivity fails. 

However, we can expand it in such a way that it does. We give two
such expansions.
\begin{example}
\label{exa:trees with predicate} Let $L_{dt}^{B}=\left\{ <,P,f,\wedge\right\} $
where $P$ is a unary predicate and $f$ is a unary function symbol,
and let $\DenseTrees^{B}$ be the model completion of the universal
$L_{dt}^{B}$-theory of trees where $P$ is a downwards closed linearly
ordered subset and $f\left(x\right)$ is the maximal element in $P$
which is $\leq x$. In other words, $\DenseTrees^{B}$ is the theory
of the Fra\"is\'e limit of the class of finite $L_{dt}^{B}$-structures
$M$ where $M\restriction\left\{ \wedge,<\right\} $ is a tree with
a meet function, $P^{M}$ is linearly ordered and downwards closed
and $f\left(x\right)=\max\set{y\leq x}{y\in P}$ (note that this class
has JEP and AP).  Then $\DenseTrees^{B}$ is the theory of dense
trees with a predicate for a branch (a maximal chain), it has quantifier
elimination and is $\omega$-categorical.  Let us see why $P$ is
a maximal chain in every model $M\models\DenseTrees^{B}$. Of course
it is downwards closed by definition, so if $a\in M$ is comparable
with $P$ but $a\notin P$, then $a>P$. As $\DenseTrees^{B}$ is
model-complete, $M$ is existentially closed (see Fact \ref{fact:Fraisse limits})
so there is some $b\in P$ (from $M)$ such that $f\left(a\right)<b$.
Thus, $a>b>f\left(a\right)$ which is a contradiction to the definition
of $f$. 

For three sets $A,B,C$, let $A\ind_{C}B$ iff $\left\langle AC\right\rangle \cap\left\langle BC\right\rangle \subseteq\left\langle C\right\rangle $
and for all $a\in\left\langle AC\right\rangle $ with $f\left(a\right)\notin\left\langle C\right\rangle $
and $b\in\left\langle BC\right\rangle $ with $f\left(b\right)\notin\left\langle C\right\rangle $
such that $f\left(a\right)\equiv_{C}f\left(b\right)$ (which is the
same as $f\left(a\right)\equiv_{f\left(C\right)}f\left(b\right)$),
$f\left(a\right)<f\left(b\right)$. Then $\ind$ is canonical.  The
only nontrivial axioms to check are stationarity over $\emptyset$,
extension and transitivity. 

Suppose that $A\ind B$. This just says that that $B$ is placed above
$A$ with respect to the branch $P$ (i.e., $f\left(a\right)<f\left(b\right)$
for all $a\in A,b\in B$). So $\ind$ is stationary by quantifier
elimination. 

Extension: suppose that $A\ind_{C}B$ and we are given a single element
$d$, which we may assume is not in $\left\langle BC\right\rangle $
and even that $f\left(d\right)\notin\left\langle BC\right\rangle $.
 First find some $d''$ such that $d''\equiv_{BC}f\left(d\right)$
and $d''$ is greater than every $f\left(a\right)$ such that $a\in\left\langle AC\right\rangle $
and $f\left(a\right)\equiv_{C}f\left(d\right)$. Then find $d'$ such
that $d'\equiv_{BC}d$ and $f\left(d'\right)=d''$ (and $\left\langle AC\right\rangle \cap\left\langle BCd'\right\rangle =\left\langle C\right\rangle $).

Transitivity: suppose that $A\ind_{DC}B$ and $D\ind_{C}B$ and we
have to show that $AD\ind_{C}B$. Suppose that $a\in\left\langle ADC\right\rangle ,f\left(a\right)\notin\left\langle C\right\rangle $
and $b\in\left\langle BC\right\rangle ,f\left(b\right)\notin\left\langle C\right\rangle $
are such that $f\left(a\right)\equiv_{C}f\left(b\right)$ but $f\left(b\right)\leq f\left(a\right)$.
Then there must be some $d\in\left\langle DC\right\rangle $ such
that $f\left(b\right)\leq d\leq f\left(a\right)$, as otherwise $f\left(a\right)\equiv_{CD}f\left(b\right)$.
But since $d\ind_{C}B$, $f\left(b\right)<d$, which implies that
$f\left(b\right)$ and $d$ do not have the same type over $C$, so
there must be some $c\in C$ between them, and in particular, it contradicts
our assumption that $f\left(a\right)\equiv_{C}f\left(b\right)$. The
other direction of transitivity is proved similarly. 
\end{example}

\begin{example}
\label{exa:trees with points} Let $L_{dt}^{p}=\left\{ <,p,\wedge\right\} $
where $p$ is a new constant. Let $\DenseTrees^{p}$ be the unique
completion of $\DenseTrees$ to $L_{dt}^{p}$. Let $M\models\DenseTrees^{p}$
be the unique countable model. To simplify notation, we identify $p$
with $p^{M}$. For three sets $A,B,C$, we let $A\ind_{C}B$ iff $\left\langle AB\right\rangle \cap\left\langle BC\right\rangle \subseteq\left\langle C\right\rangle $
and: 
\begin{enumerate}
\item For all $a\in\left\langle AC\right\rangle $ with $a\wedge p\notin\left\langle C\right\rangle $
and $b\in\left\langle BC\right\rangle $ with $b\wedge p\notin\left\langle C\right\rangle $
such that $a\wedge p\equiv_{C}b\wedge p$, $a\wedge p<b\wedge p$.
\item For all $a\in\left\langle AC\right\rangle $ such that $a>p$ with
no $c\in\left\langle C\right\rangle $ such that $a\wedge c>p$, and
all $b\in\left\langle BC\right\rangle $ with $b>p$ and no $c\in\left\langle C\right\rangle $
such that $b\wedge c>p$, $a\wedge b=p$.
\end{enumerate}
Then $\ind$ is canonical. The only nontrivial axioms to check are
stationarity over $\emptyset$, extension and transitivity. 

It is stationary over $\emptyset$ by elimination of quantifiers,
since $A\ind B$ iff $A'=\set{a\in A}{a\wedge p<p}$ is placed below
$B'=\set{b\in B}{b\wedge p<p}$ with respect to the points below $p$
while $A''=\set{a\in A}{a\geq p}$ and $B''=\set{b\in B}{b\geq p}$
are placed independently above $p$. 

Extension: suppose that $A\ind_{C}B$ and we are given $d$ such that
$d\notin\left\langle BC\right\rangle $. First assume that $d\wedge p<p$.
If $d\wedge p\notin\left\langle BC\right\rangle $, similarly to Example
\ref{exa:trees with predicate}, first find some $d''$ such that
$d''\equiv_{BC}d\wedge p$ and $d''>a\wedge p$ for all $a\in\left\langle AC\right\rangle $
with $a\wedge p\equiv_{C}d\wedge p$. Then find $d'$ such that $d'\equiv_{BC}d$
with $d'\wedge p\equiv d''$ (and $\left\langle BCd'\right\rangle \cap\left\langle AC\right\rangle =\left\langle C\right\rangle $).
 Now assume that $d>p$. If there is some $b\in\left\langle BC\right\rangle $
with $b\wedge d>p$, any $d'\equiv_{BC}d$ such that $\left\langle BCd'\right\rangle \cap\left\langle AC\right\rangle =\left\langle C\right\rangle $
will work. Otherwise find some $d'\equiv_{BC}d$ such that $d'\wedge a=p$
for all $a\in\left\langle AC\right\rangle $ with $a>p$. 

Transitivity: suppose that $A\ind_{DC}B$ and $D\ind_{C}B$ and we
have to show that $AD\ind_{C}B$. If we are in case (1) of the definition,
($a\in\left\langle ADC\right\rangle $, $a\wedge p\notin\left\langle C\right\rangle $,
etc.) then we proceed exactly as in Example \ref{exa:trees with predicate}.
Otherwise, suppose that $a\in\left\langle ADC\right\rangle $, $b\in\left\langle BC\right\rangle $
are as in case (2). If there is some $d\in\left\langle CD\right\rangle $
with $a\wedge d>p$, then for no $c\in\left\langle C\right\rangle $
is it the case that $c\wedge d>p$ (otherwise $a\wedge c>p$). Thus
$d\wedge b=p$ because $D\ind_{C}B$ hence $a\wedge b=p$ as required.
If there is no such $d$ then $a\wedge b=p$ because $A\ind_{DC}B$. 
\end{example}

Trees also satisfy the following interesting phenomenon.
\begin{prop}
\label{prop:in trees, every automorphism fixes a branch or a point}If
$M$ is a dense tree as in Example \ref{exa:dense trees} (i.e., the
model companion of the theory of trees in $\left\{ <,\wedge\right\} $)
then for every $\sigma\in\Aut\left(M\right)$ which does not have
any fixed points, there is a branch $B\subseteq M$ such that $\sigma\left(B\right)=B$. 
\end{prop}

\begin{proof}
Let $B$ be a maximal linearly ordered set such that $\sigma\left(B\right)=B$
(which exists by Zorn's lemma). We will show that $B$ is a branch.
Note that if $x\in B$ and $y<x$, then $B\cup\set{\sigma^{n}\left(y\right)}{n\in\Zz}$
is still a chain: given any $z\in B$ and any $n\in\Zz$, $\sigma^{n}\left(x\right),z$
are comparable and $\sigma^{n}\left(y\right)<\sigma^{n}\left(x\right)$
it follows that $\sigma^{n}\left(y\right)$ and $z$ are comparable
(if $z\leq\sigma^{n}\left(x\right)$ then both $\sigma^{n}\left(y\right),z\leq\sigma^{n}\left(x\right)$,
so they are comparable by the tree axioms, and if $\sigma^{n}\left(x\right)<z$,
then $\sigma^{n}\left(y\right)<z$), and for any $n,m\in\Zz$, $\sigma^{n}\left(y\right),\sigma^{m}\left(y\right)$
are comparable since $\sigma^{n}\left(x\right)$ and $\sigma^{m}\left(x\right)$
are (if $\sigma^{n}\left(x\right)\leq\sigma^{m}\left(x\right)$ then
both $\sigma^{n}\left(y\right),\sigma^{m}\left(y\right)\leq\sigma^{m}\left(x\right)$
so they are comparable by the tree axioms). Hence $B$ is downwards
closed. 

Now, as $\sigma$ has no fixed points, $B$ cannot have a maximum
(which would have to be a fixed point). Also, if $a\geq B$ and $\sigma\left(a\right)\geq a$
or $\sigma\left(a\right)\leq a$ then $B\cup\set{\sigma^{n}\left(a\right)}{n\in\Zz}$
is still a chain (since $\sigma^{n}\left(a\right)\geq\sigma^{n}\left(B\right)=B$
for all $n\in\Zz$), so $a\in B$.

If $B$ is not a branch (in particular, if $B=\emptyset$, which we
haven't ruled out yet), there is some $a\in M$ such that $B<a$.
Let $b=\sigma\left(a\right)\neq a$ (and by the above, $b,a$ are
not comparable), so $B<b$. Hence $B\leq\left(a\wedge b\right)<a,b$.
Now, $\sigma\left(a\wedge b\right)<\sigma\left(a\right)=b$, so $a\wedge b$
and $\sigma\left(a\wedge b\right)$ are comparable. The previous paragraph
implies that $a\wedge b\in B$. But then $B$ has a maximum \textemdash{}
contradiction. 
\end{proof}

\section{Having finite topological rank }

In this section we will find some criteria that ensure that $G$ has
finite topological rank.

\subsection{$\omega$-categorical stable theories }
\begin{prop}
If $T$ is stable $\omega$-categorical, $M\models T$ is countable
and $\Aut\left(\acl^{\eq}\left(\emptyset\right)\right)$ is finite,
then $\Aut\left(M\right)$ has finite topological rank. 
\end{prop}

\begin{proof}
Without loss of generality, $M=M^{\eq}$ (if $S\subseteq\Aut\left(M^{\eq}\right)$
generates a dense subgroup, then $S\restriction M=\set{f\restriction M}{f\in S}$
generates a dense subgroup of $\Aut\left(M\right)$). Let $N=M_{\acl\left(\emptyset\right)}$
(i.e., name the elements in $\acl\left(\emptyset\right)$). Then $N$
is $\omega$-categorical by Propsition \ref{prop:G0 acts oligomorphically and has no CQ}.
Then in $N$, $\acl^{\eq}\left(\emptyset\right)=\dcl^{\eq}\left(\emptyset\right)$,
so by Example \ref{exa:T stable with emptyset stationary}, there
is a canonical independence relation in $N$, so $G^{0}=\Aut\left(N\right)$
is topologically 2-generated by Corollary \ref{cor:having a groovy implies dense},
say by $\left\{ f_{1},f_{2}\right\} $. Now, $\Aut\left(M\right)/G^{0}$
is finite by assumption, so let $S\subseteq\Aut\left(M\right)$ be
a finite set of representatives. Then $S\cup\left\{ f_{1},f_{2}\right\} $
generates a dense subgroup $\Aut\left(M\right)$: given two finite
tuples $\bar{a},\bar{b}$ from $M$ such that $\bar{a}\equiv\bar{b}$,
there is an automorphism $\sigma\in\Aut\left(M\right)$ such that
$\sigma\left(\bar{a}\right)=\bar{b}$. Also, there is some $f\in S$
such that $f^{-1}\sigma\in\Aut\left(N\right)$. Hence for some $g$
in the group generated by $\left\{ f_{1},f_{2}\right\} $, $g\left(\bar{a}\right)=f^{-1}\sigma\left(\bar{a}\right)=f^{-1}\left(\bar{b}\right)$,
so $fg\left(\bar{a}\right)=\bar{b}$.
\end{proof}
The following fact implies immediately the next result. 
\begin{fact}
\cite[Lemma 3.1]{MR1226301}\label{fact:finite acl} If $T$ is $\omega$-categorical
and $\omega$-stable and $M\models T$ is countable, then $\Aut\left(\acl^{\eq}\left(\emptyset\right)\right)$
is finite. 
\end{fact}

\begin{cor}
If $T$ is $\omega$-stable and $\omega$-categorical and $M\models T$
is countable, then $\Aut\left(M\right)$ has finite topological rank.
\end{cor}

\subsection{Reducing finite topological rank to expansions}

Suppose that $M$ is countable let $G=\Aut\left(M\right)$. We now
want to explore the idea that perhaps by expanding $M$ (i.e., moving
to a subgroup), we can show that the topological rank of $G$ is small
by showing that the rank of the automorphism group of the expansion
is. Suppose that $H\leq G$. If $\left(G,H\right)$ has a compact
quotient (see Definition \ref{def:a pair having a CQ}), then we cannot
hope to deduce anything. For example, by Proposition \ref{prop:G0 acts oligomorphically and has no CQ}
we have that $G^{0}$ acts oligomorphically on $M$ and it can be
that $G^{0}$ has a cyclically dense conjugacy class (so topological
rank 2) while $G/G^{0}=\left(\Zz/2\Zz\right)^{\omega}$ (so $G$ is
not topologically finitely generated) \textemdash{} this happens in
the example described in in Remark \ref{rem:constraint for cdcc compact quotient},
see Example \ref{exa:naming the equivalence classes has SAP}.  Indeed,
we will see that $\left(G,H\right)$ having a compact quotient is
the only obstruction. 

\subsubsection{$\omega$-categorical structures with finitely many reducts}
\begin{thm}
\label{thm:Finitely many reducts and empty acl}Suppose that $H\leq G$
is closed  and that $\left(G,H\right)$ has no compact quotients.
If there are only finitely many closed groups between $G$ and $H$
then there is some $g\in G$ such that $H\cup\left\{ g\right\} $
topologically generate $G$. 
\end{thm}

\begin{rem}
The condition of having finitely many closed groups in the theorem
holds when for instance $M$ is a reduct of an $\omega$-categorical
structure $M'$ where $H=\Aut\left(M'\right)$, and $M'$ has only
finitely many reducts up to bi-definability. 
\end{rem}

\begin{proof}
Let $\set{H_{i}}{i<n}$ be the family of closed proper subgroups of
$G$ containing $H$ (which is finite by assumption). If $\left[G:H_{i}\right]<\infty$
for some $i<n$, then there would be a closed normal proper subgroup
$N_{i}\trianglelefteq G$ of finite index such that $N_{i}\leq H_{i}$
(in general, if $H'\leq G$ is closed of finite index, then there
is a closed normal subgroup $N\leq H'$, $N\trianglelefteq G$ such
that $\left[G:N\right]<\infty$. In fact, $N=\bigcap\set{gH'g^{-1}}{g\in G}$
and this intersection is finite as it is the orbit of $H'$ under
the action of $G$ on conjugates of $H'$ and its stabilizer contains
$H'$).  But then $N_{i}H=G$ by assumption and Proposition \ref{prop:no compact quotients},
so $G=N_{i}H\subseteq H_{i}H=H_{i}$ contradicting the fact that $H_{i}$
was a proper subgroup. 

By a theorem of Neumann \cite[Lemma 4.1]{MR0062122}, there is some
$g\in G\backslash\bigcup\set{H_{i}}{i<n}$. If $G\neq\cl\left(\left\langle H\cup\left\{ g\right\} \right\rangle \right)$
(the topological closure of the group generated by $H\cup\left\{ g\right\} $),
then $\cl\left(\left\langle H\cup\left\{ g\right\} \right\rangle \right)$
is one of the groups $H_{i}$, contradicting the choice of $g$. 
\end{proof}
\begin{cor}
\label{cor:finite topological rank from finitely many reducts}If
$G$ and $H$ are as in Theorem \ref{thm:Finitely many reducts and empty acl}
and $H$ has finite topological rank then so does $G$.
\end{cor}

By Example \ref{exa:If G0=00003DG, then no compact quotients}, in
the $\omega$-categorical context we get that if $\acl^{\eq}\left(\emptyset\right)=\dcl^{\eq}\left(\emptyset\right)$
in $M$ and $M'$ is an expansion having finitely many reducts, then
we can apply Corollary \ref{cor:finite topological rank from finitely many reducts}.
This is the case, for instance, when $M'$ is $\left(\Qq,<\right)$
(see \cite{DBLP:journals/jsyml/JunkerZ08}). By Lemma \cite[Lemma 2.10]{DBLP:journals/jsyml/JunkerZ08},
an example of such a reduct of DLO is given by the countable dense
circular order, which is the structure with universe $\Qq$, and a
ternary relation $C\left(x,y,z\right)$ given by $C\left(x,y,z\right)\Leftrightarrow x<y<z\vee y<z<x\vee z<x<y$.
\begin{cor}
\label{cor:Circular order has no CIR but degree <=00003D3}$\Aut\left(\Qq,C\right)$
has topological rank $\leq3$, but $\left(\Qq,C\right)$ has no CIR. 
\end{cor}

\begin{proof}
We only have to show that it has no CIR. By Lemma \ref{lem:CIR implies just over 0},
if there was a CIR, then in particular there would be a type of a
single element $q\left(x\right)$ over $\Qq$ which does not split
over $\emptyset$. But by quantifier elimination, every tuple of two
distinct elements have the same type (i.e., $\Aut\left(\Qq,C\right)$
acts 2-transitively on $\Qq$). Now, $q$ cannot be realized in $\Qq$
and must contain $C\left(0,x,1\right)$ or $C\left(1,x,0\right)$,
hence both, which is a contradiction. 
\end{proof}
\begin{rem}
\label{rem:expanding a circular order by a point has a CIR}For any
point $a\in\Qq$, the expansion $\left(\Qq,C,a\right)$ does have
a CIR. Indeed, in this case $C$ defines a dense linear order with
no endpoints on $\Qq\backslash\left\{ a\right\} $ by $b<c\iff C\left(a,b,c\right)$.
Since $\left(\Qq,<\right)$ has a CIR $\ind$ (Example \ref{exa:DLO has groovy}),
we can define $A\ind_{C}^{*}B$ by $A\backslash\left\{ a\right\} \ind_{C\backslash\left\{ a\right\} }B\backslash\left\{ a\right\} $.
Since for every finite tuples $b,c$, $b\equiv c$ in the expansion
iff $b\backslash a\equiv c\backslash a$ in the order, it follows
easily that $\ind^{*}$ is a CIR. 
\end{rem}

An even closer look at the reducts of DLO, gives the following result.
\begin{cor}
Every closed supergroup of $\Aut\left(\Qq,<\right)$ has topological
rank $\leq3$. 
\end{cor}

\begin{proof}
The diagram in \cite[page  867]{DBLP:journals/jsyml/JunkerZ08} of
the lattice of closed groups between $\Aut\left(\Qq,<\right)$ and
$\Aut\left(\Qq,=\right)$ shows that any such group contains at most
two incomparable closed subgroups. Since no group can be a union of
two of its proper subgroups, we do not need to use Neumanns's lemma
in the proof of Theorem \ref{thm:Finitely many reducts and empty acl}
above, allowing us to drop the assumption that $\left(G,H\right)$
has no compact quotients. 
\end{proof}

\subsubsection{A general reduction theorem}

In the next theorem we drop the assumption of having finitely many
reducts of the expansion (i.e., of having finitely many groups between
$H$ and $G$), but we compensate for it by assuming that $H$ acts
oligomorphically on $M$ and increasing the number of generators by
1. 
\begin{fact}
\label{fact:Udi David Evans on acl}\cite[Lemma 1.4]{MR1226301} Suppose
that $M$ is a countable $\omega$-saturated structure. Then for any
$A,B\subseteq M$, there is some $A'$ (in the monster model $\C$,
see just above Section \ref{subsec:A-mix-of}) such that $A'\equiv A$
and $A'\cap B\subseteq\acl\left(\emptyset\right)$.
\end{fact}

\begin{thm}
\label{thm:If no CQ then finitely generated}Suppose as usual that
$M$ is countable and $\omega$-categorical and let $G=\Aut\left(M\right)$.
Suppose that $H\leq G$ is closed and acts oligomorphically on $M$
and that $\left(G,H\right)$ has no compact quotients. Then there
are $g_{1},g_{2}\in G$ such that $H\cup\left\{ g_{1},g_{2}\right\} $
topologically generates $G$. 
\end{thm}

\begin{proof}
Let $M'$ be an $\omega$-categorical expansion of $M$ to some language
$L'$ containing $L$ (the language of $M$) such that $H=\Aut\left(M'\right)$.
We use $'$ to indicate the expansion. In particular, $\C'$ denotes
the expansion of $\C$ to $L'$. 

By Fact \ref{fact:Udi David Evans on acl}, there is some $M_{0}$
such that $M_{0}\equiv M$ and $M_{0}^{\eq}\cap M^{\eq}=\acl^{\eq}\left(\emptyset\right)$
(apply the fact in $\C^{\eq}$). There is some automorphism $\sigma$
of $\C$ such that $\sigma\left(M_{0}\right)=M$. Let $N_{0}'$ be
a countable model containing $\sigma^{n}\left(M_{0}\right)$ for all
$n\in\Zz$. Let $N_{1}'$ be a countable model containing $\sigma^{n}\left(N_{0}'\right)$
for all $n\in\Zz$. Continue like this and finally let $N_{\omega}'=\bigcup\set{N_{i}'}{i<\omega}$.
So $M'\prec N'_{\omega}\prec\C'$ is countable and $\sigma\restriction N{}_{\omega}\in\Aut\left(N_{\omega}\right)$.
By $\omega$-categoricity (of $M'$) we may assume that $N_{\omega}'=M'$:
there is some $g_{1}\in\Aut\left(M\right)$ and $M_{0}'\prec M'$
such that $g_{1}\left(M_{0}^{\eq}\right)\cap M_{0}^{\eq}=\acl^{\eq}\left(\emptyset\right)$.

Then $H_{1}=\cl\left(\left\langle H,g_{1}\right\rangle \right)$ is
a closed group acting oligomorphically on $M$. Also, note that $\left(G,H_{1}\right)$
has no compact quotients. Let $M''$ be the reduct of $M'$, which
is also an expansion of $M$ that corresponds to $H_{1}$: $\Aut\left(M''\right)=H_{1}$.
As usual, we use $''$ to indicate that we work in this expansion.
\begin{claim}
\label{claim:property of M''}If $X\subseteq M^{n}$ is definable
over $\emptyset''$ (i.e., definable in $L''$ over $\emptyset$)
and $M$-definable (in $L$), then it is $\emptyset$-definable (in
$L$).
\end{claim}

\begin{proof}
First note that  it is enough to show that $X$ is $\acl_{L}^{\eq}\left(\emptyset\right)$-definable
(the code $\code X$ of $X$ belongs to $\dcl_{L''}^{\eq}\left(\emptyset\right)$
and to $\acl_{L}^{\eq}\left(\emptyset\right)$, and if it were not
in $\dcl_{L}^{\eq}\left(\emptyset\right)$ then there would be an
automorphism of $M$ moving it, but then by the no-compact quotient
assumption there would be an automorphism of $M''$ moving it as well
\textemdash{} contradiction). 

Now, since $X$ is $\emptyset''$-definable and $M$-definable, it
is definable over $M_{0}$ (because $M_{0}'\prec M'$), so its code
$\code X\in M_{0}^{\eq}$. In addition, $g_{1}\left(X\right)=X$,
so $X$ is definable over $g_{1}\left(M_{0}\right)$, hence $\code X\in g_{1}\left(M_{0}^{\eq}\right)$.
Together it is in $\acl^{\eq}\left(\emptyset\right)$, which is what
we wanted. 
\end{proof}
Now we construct $g_{2}$ by back-and-forth to ensure that $\cl\left(\left\langle H_{1},g_{2}\right\rangle \right)=G$. 

Suppose that we have constructed $g_{2}\restriction A$ for some finite
set $A$. Let $O$ be an orbit of the action of $G$ on $M^{m}$,
and we write it as $O=\bigcup\set{O_{i}}{i<n}$ where the $O_{i}$'s
are the orbits of the action of $H_{1}$ (recall that $H_{1}$ acts
oligomorphically on $M$, so there are only finitely many such orbits). 
\begin{claim}
\label{claim:extending g_2}For any subset $s\subsetneq n$ there
are $a,b\in O$ such that $a\in O_{s}=\bigcup\set{O_{i}}{i\in s},b\in O_{n\backslash s}$,
and $g_{2}\restriction A\cup\left\{ \left\langle a,b\right\rangle \right\} $
or $g_{2}\restriction A\cup\left\{ \left\langle b,a\right\rangle \right\} $
is an elementary map. 
\end{claim}

\begin{proof}
Note that $O_{s}$ is $\emptyset''$-definable. As it is not $\emptyset$-definable
(because $s\subsetneq n$), it is also not $M$-definable by Claim
\ref{claim:property of M''}. In particular, it is not $A$-definable.
Hence there are $a_{0}\in O_{s},a_{1}\in O_{n\backslash s}$ such
that $a_{0}\equiv_{A}a_{1}$. There is some $b$ such that $a_{0}A\equiv a_{1}A\equiv bg_{2}\left(A\right)$.
If $b\in O_{s}$, then $g_{2}\upharpoonright A\cup\left\{ \left\langle a_{1},b\right\rangle \right\} $
is the required map. Otherwise, pick $g_{2}\restriction A\cup\left\{ \left\langle a_{0},b\right\rangle \right\} $. 
\end{proof}
In the back-and-forth construction of $g_{2}$, we deal with all these
orbits (for every $m<\omega$, there are only finitely many) and all
these subsets $s$ and increase $g_{2}$ according to Claim \ref{claim:extending g_2}.
We claim that $g_{2}$ is such that $\cl\left(\left\langle H_{1},g_{2}\right\rangle \right)=G$.
Indeed, it is enough to show that every orbit $O$ of $G$ is also
an orbit of $\left\langle H_{1},g_{2}\right\rangle $. The orbit $O$
can be written as $\bigcup\set{O_{i}}{i<n}$ where the $O_{i}$'s
are the orbits of $H_{1}$, and also as $\bigcup\set{O'_{i}}{i\in I}$
where the $O'_{i}$'s are orbits of $\left\langle H_{1},g_{2}\right\rangle $.
Each such $O_{i}'$ is itself a union of $H_{1}$-orbits, so has the
form $O_{s}$ for some $s\subseteq n$. But by construction, if $s\neq n$
there are tuples $a\in O_{s},b\in O_{n\backslash s}$ such that either
$g_{2}$ or $g_{2}^{-1}$ maps $a$ to $b$ \textemdash{} contradiction.
So $s=n$, and $O_{i}'=O$. 
\end{proof}

\section{\label{sec:A-topological-dynamics}A topological dynamics consequence
of having a CIR }
\begin{defn}
Suppose that $M$ is a countable structure. Call an automorphism $\sigma\in G$\emph{
shifty} if there is some invariant binary relation on finite sets
in $M$, $\ind$ (the base will always be $\emptyset$) such that:
\begin{itemize}
\item (Monotonicity) If $A\ind B$ and $A'\subseteq A$, $B'\subseteq B$
then $A'\ind B'$.
\item (Right existence) For every finite tuple $a$ there is some $a'\equiv a$
such that $a\ind a'$ (by this we mean that sets enumerated by $a$,
$a'$ are independent). 
\item (Right shiftiness) If $A$ is finite and $b,b'$ are finite tuples
such that $b'\equiv b$ and $A\ind b'$, then there exists some $n<\omega$
such that  $b'\equiv_{A}\sigma^{n}\left(b\right)$. 
\end{itemize}
\end{defn}

\begin{lem}
If $\sigma$ is shifty then it also satisfies: 
\begin{itemize}
\item (Left existence) For every finite tuple $a$ there is some $a'\equiv a$
such that $a'\ind a$.
\item (Left shiftiness) If $A$ is finite and $b,b'$ are finite tuples
such that $b'\equiv b$ and $b'\ind A$, then there exists some $n<\omega$
such that  $b'\equiv_{A}\sigma^{-n}\left(b\right)$. 
\end{itemize}
\end{lem}

\begin{proof}
Suppose that $\sigma$ is shifty, as witnessed by $\ind$. Given $a$,
there is some $a'\equiv a$ such that $a\ind a'$. Applying an automorphism
taking $a'$ to $a$ we get some $a''\equiv a$ such that $a''\ind a$,
which shows left existence. 

As for left shiftiness, suppose that $A$ is finite and enumerated
by $a$, $b,b'$ are finite tuples such that $b'\ind A$ and $b\equiv b'$.
Then applying an automorphism, we get some $a'$ such that $ab'\equiv a'b$,
so $b\ind a'$. Hence for some $n<\omega$, $a'\equiv_{b}\sigma^{n}\left(a\right)$.
From $a'b\equiv\sigma^{n}\left(a\right)b$ we get that $ab'\equiv a'b\equiv\sigma^{-n}\left(a'\right)\sigma^{-n}\left(b\right)\equiv a\sigma^{-n}\left(b\right)$,
i.e., $b'\equiv_{A}\sigma^{-n}\left(b\right)$. 
\end{proof}
\begin{prop}
The automorphism $\sigma$ is a shifty automorphism on $M$ iff for
any type $p\in S\left(\emptyset\right)$ (with finitely many variables),
letting $Y_{a}=\bigcap\set{\bigcup\set{\tp\left(a,\sigma^{n}\left(a'\right)\right)}{n<\omega}}{a'\equiv a}$
for any $a\models p$, the intersection $Y_{p}=\bigcap\set{Y_{a}}{a\models p}$
is nonempty.
\end{prop}

\begin{proof}
Suppose that $\sigma$ is shifty, and fix some type $p\in S\left(\emptyset\right)$.
Let $a\models p$. By existence, there is some $a'\equiv a$ with
$a\ind a'$. Let $q=\tp\left(a,a'\right)$ and fix some $b\models p$.
Let $\tau\in\Aut\left(M\right)$ map $a$ to $b$ and let $b'=\tau\left(a'\right)$.
We have that $b\ind b'$ and hence by right shiftiness, $q=\tp\left(b,b'\right)\in Y_{b}$.
Since $b$ was arbitrary, $q\in Y_{p}$. 

Suppose that the right hand side holds. Given a finite tuple $a$
and $a'\equiv a$, write $a\ind^{*}a'$ iff $\tp\left(a,a'\right)\in Y_{p}$
where $p=\tp\left(a/\emptyset\right)$. For general finite sets $A,B$,
write $A\ind B$ iff there is some $C$ containing $A$ and $C'$
containing $B$ such that $C\ind^{*}C'$. Obviously, $\ind$ is invariant
and monotone. Right existence follows from the assumption that $Y_{p}\neq\emptyset$
for all $p\in S\left(\emptyset\right)$. Right shiftiness: suppose
that $a\ind^{*}a'$ and $a''\equiv a'$. Then $\tp\left(a,a'\right)\in Y_{p}$
and in particular it belongs to $Y_{a}$. By definition of $Y_{a}$,
$\tp\left(a,a'\right)\in\bigcup\set{\tp\left(a,\sigma^{n}\left(a''\right)\right)}{n<\omega}$,
so for some $n<\omega$, $aa'\equiv a_{i}\sigma^{n}\left(a''\right)$. 
\end{proof}
\begin{prop}
\label{prop:CIR implies shifty}If $M$ is an ultrahomogeneous structure
and $\ind$ is a CIR on finite subsets of $M$ which respects substructures,
then there exists a shifty automorphism $\sigma$ on $M$, as witnessed
by $\ind$.
\end{prop}

\begin{proof}
Monotonicity and right existence are parts of the properties of a
CIR, so we only have to prove right shiftiness. Suppose that $A\ind b$
and $b'\equiv b$. By the proof of Theorem \ref{thm:existence of repulsive automorphism-ultrahomogeneous},
the repulsive automorphism $\sigma$ constructed there satisfies that
for some $n<\omega$, $A\ind\sigma^{n}\left(b'\right)$. By stationarity,
$b\equiv_{A}\sigma\left(b'\right)$. 
\end{proof}
Recall the definitions of flow and subflow from Section \ref{subsec:Ramsey, top dynamics}.

\begin{thm}
\label{thm:contains a subflow}Let $M$ be a countable homogeneous
structure and $G=\Aut\left(M\right)$. 

Suppose that $\sigma\in G$ is a shifty automorphism and that $\left(X,d\right)$
is a compact \uline{metric} $G$-flow. Then for every $x_{*}\in X$
there is some conjugate $\sigma_{*}\in G$ of $\sigma$ such that:
\begin{itemize}
\item [(*)] Both $\cl\set{\sigma_{*}^{n}\left(x_{*}\right)}{n<\omega}$
and $\cl\set{\sigma_{*}^{-n}\left(x_{*}\right)}{n<\omega}$ contain
a subflow of $X$. 
\end{itemize}
\end{thm}

\begin{rem}
Note that Theorem \ref{thm:contains a subflow} implies that both
$\bigcap\set{\cl\set{\sigma_{*}^{n}\left(x_{0}\right)}{k\leq n<\omega}}{k<\omega}$
and $\bigcap\set{\cl\set{\sigma_{*}^{-n}\left(x_{0}\right)}{k\leq n<\omega}}{k<\omega}$
contain a subflow of $X$: if e.g., $Y_{0}$ is a flow contained in
the left space, then $GY_{0}=Y_{0}$, so $\sigma_{*}^{-k}\left(Y_{0}\right)\subseteq Y_{0}\subseteq\cl\set{\sigma_{*}^{n}\left(x_{*}\right)}{n<\omega}$,
hence $Y_{0}\subseteq\sigma_{*}^{k}\left(\cl\set{\sigma_{*}^{n}\left(x_{*}\right)}{n<\omega}\right)=\cl\set{\sigma_{*}^{n}\left(x_{*}\right)}{k\leq n<\omega}$. 
\end{rem}

Before the proof we note the following useful lemma. 
\begin{lem}
\label{lem:uinform bound}Suppose that $G$ is a topological group
acting continuously on a compact metric space $\left(X,d\right)$.
Then for every $0<\varepsilon$ there is some open neighborhood $U$
of $\id\in G$ such that for every $g,h\in G$ if $gh^{-1}\in U$
then for all $x\in X$ we have that $d\left(gx,hx\right)<\varepsilon$. 
\end{lem}

\begin{proof}
It is enough to show that there is some open neighborhood $U$ of
$\id$ such that if $g\in U$ then for all $x\in X$, $d\left(gx,x\right)<\varepsilon$
(since then if $gh^{-1}\in U$ then $d\left(gh^{-1}\left(hx\right),hx\right)<\varepsilon$).
For every $x\in X$, there is some neighborhood $V_{x}$ of $x$ in
$X$ and some neighborhood $U_{x}$ of $\id$ in $G$ such that for
all $g\in U_{x}$, $x'\in V_{x}$, $d\left(gx',x'\right)<\varepsilon$.
By compactness, a finite union of $V_{x}$'s covers $X$. Let $U$
be the intersection of the corresponding $U_{x}$'s. 
\end{proof}

\begin{proof}[Proof of Theorem \ref{thm:contains a subflow}]
Suppose that $\ind$ witnesses that $\sigma$ is shifty. Let $G_{0}$
be a countable dense subset of $G$, enumerated as $\sequence{g_{i}}{i<\omega}$,
such that $g_{0}=\id$. 

We construct an automorphism $\tau:M\to M$ by back and forth such
that eventually $\sigma_{*}=\tau^{-1}\sigma\tau$ and such that at
each finite stage, $\tau$ will be an elementary map. For the construction
it is actually better to think of the domain and range of $\tau$
as two different structures, so we have $M=M_{*}$ and suppose that
$\sigma:M\to M$, $\sigma_{*}:M_{*}\to M_{*}$ and $\tau:M_{*}\to M$.
The subscript $*$ will denote tuples from $M_{*}$ throughout. 

Suppose that we have constructed a partial elementary map $f:A_{*}\to A$
(that will be part of $\tau$ eventually) with $A_{*}\subseteq M_{*},A\subseteq M$
finite, enumerated by $a_{*},a$. Here is the main tool in the construction.
\begin{claim}
\label{claim:Using shiftiness}Suppose that $b_{*}'\ind a_{*}$ and
$b_{*}'\equiv b\subseteq a$. Let $b_{*}=f^{-1}\left(b\right)$. Then
there is $k<\omega$ and an extension $f'$ of $f$ such that any
automorphism $\tau'$ extending $f'$ will satisfy that for $\sigma_{*}'=\tau'^{-1}\sigma\tau'$,
$\sigma_{*}'^{k}\left(b'_{*}\right)=b_{*}$. 

Similarly, if $a_{*}\ind b_{*}'$ then there is some $k<\omega$ and
an extension $f'$ of $f$ such that any automorphism $\tau'$ extending
$f'$ will satisfy that for $\sigma_{*}'=\tau'^{-1}\sigma\tau'$,
$\sigma_{*}'^{k}\left(b_{*}\right)=b'_{*}$. 
\end{claim}

\begin{proof}
First, find some tuple $b'$ in $M$ such that $b'a\equiv b_{*}'a_{*}$.
In particular, $b'\ind a$. By left shiftiness, there is some $k<\omega$
such that $\sigma^{-k}\left(b\right)a\equiv b'a\equiv b_{*}'a_{*}$.
Extend $f$ to $f'$ which sends $b_{*}'$ to $\sigma^{-k}\left(b\right)$.
Then, for any $\tau'$ extending $f'$, $\tau'^{-1}\sigma{}^{k}\tau'\left(b_{*}'\right)=\tau'^{-1}\left(b\right)=b_{*}$. 

The second statement is proved similarly, using right shiftiness.
\end{proof}
We will make sure that for each $n<\omega$, the following condition
holds.
\begin{itemize}
\item [$\star$]There are $k_{n,0},\ldots,k_{n,n-1}<\omega$ such that for
all $i<n$, $d\left(\sigma_{*}^{k_{n,i}}\left(x_{*}\right),g_{i}\left(\sigma_{*}^{k_{n,0}}\left(x_{*}\right)\right)\right)<1/n$
and $k_{n,0}',\ldots,k_{n,n-1}'<\omega$ such that for all $i<n$,
$d\left(\sigma_{*}^{-k_{n,i}'}\left(x_{*}\right),g_{i}\left(\sigma_{*}^{-k_{n,0}'}\left(x_{*}\right)\right)\right)<1/n$. 
\end{itemize}
Why is $\star$ enough? Let $y_{n}=\sigma_{*}^{k_{n,0}}\left(x_{*}\right)$,
and let $y$ be a limit of some subsequence $\sequence{y_{n_{j}}}{j<\omega}$
(which exists by compactness), then $Gy\subseteq\cl\set{\sigma_{*}^{n}\left(x_{*}\right)}{n<\omega}$
(so $\cl\left(Gy\right)$ is a subflow): given $g\in G$ and $0<\varepsilon$,
first find an open neighborhood $U\subseteq G$ of $g$ such that
if $h\in U$ then $d\left(gx,hx\right)<\varepsilon/4$ for all $x\in X$
(this $U$ is given to us by Lemma \ref{lem:uinform bound}: it is
$\left(g^{-1}V\right)^{-1}$ where $V$ is an open neighborhood of
$\id$ such that if $gh^{-1}\in V$, $d\left(gx,hx\right)<\varepsilon/4$).
Take $n$ so large that $g_{i}\in U$ for some $i<n$ and $1/n<\varepsilon/4$,
and find $n_{j}$ even larger so that $d\left(g_{i}y,g_{i}y_{n_{j}}\right)<\varepsilon/4$.
Then $d\left(gy,g_{i}y\right)<\varepsilon/4$, $d\left(g_{i}y,g_{i}y_{n_{j}}\right)<\varepsilon/4$
and $d\left(g_{i}y_{n_{j}},\sigma_{*}^{k_{n_{j},i}}\left(x_{*}\right)\right)<\varepsilon/4$.
Together, $d\left(gy,\sigma_{*}^{k_{n_{j},i}}\left(x_{0}\right)\right)<3\varepsilon/4<\varepsilon$,
which means that $gy$ is in the closure. Similarly, if $y'$ is a
limit of a subsequence of $\sigma_{*}^{-k_{n,0}'}\left(x_{*}\right)$,
then $Gy'\subseteq\cl\set{\sigma_{*}^{-n}\left(x_{*}\right)}{n<\omega}$. 

So we consider $f:A_{*}\to A$ a partial elementary map. Our task
now is to deal with $n<\omega$ . Let $\varepsilon=1/n$. 

Let $A_{*}\subseteq C_{*}\subseteq M_{*}$ be finite such that if
$g^{-1}\restriction C_{*}=h^{-1}\restriction C_{*}$ then $d\left(gx,hx\right)<\varepsilon/4$
for all $x\in X$ and any $g,h\in G$ (this is by Lemma \ref{lem:uinform bound}).
Let $z_{0},\ldots,z_{l-1}$ be such that $\bigcup\set{B\left(z_{j},\varepsilon/4\right)}{j<l}$
cover $X$, and write $B_{j}=B\left(z_{j},\varepsilon/4\right)$. 

Let $c_{*}$ be a finite tuple enumerating $C_{*}$. For every $c_{*}'\equiv c_{*}$,
we say that $c_{*}'$ \emph{has} \emph{color $j<l$} if $j$ is least
such that there is $g\in G$ such that $g\left(c_{*}'\right)=c_{*}$
and $gx_{*}\in B_{j}$. Note that by the choice of $c_{*}$, if $g'\left(c_{*}'\right)=c_{*}$
then $g'g^{-1}\restriction C_{*}=\id$, so $g'x_{0}\in B\left(z_{j},\varepsilon/2\right)$. 

Let $D_{*}=\bigcup\set{g_{i}^{-1}\left(C_{*}\right)}{i<n}$. Note
that $C_{*}\subseteq D_{*}$ because $g_{0}=\id$. Let $d_{*}$ enumerate
$D_{*}$. For any $d_{*}'\equiv d_{*}$ and $s\subseteq l$, we say
that $d_{*}'$ \emph{has} \emph{color} $s$ if $\set{j<l}{c_{*}'\equiv c_{*},c_{*}'\subseteq d_{*}',c_{*}'\text{ has color }j}=s$. 

By left existence, there is some $s_{0}\subseteq l$ such that for
\emph{every} finite set $S\subseteq M_{*}$, there is some $d_{*}'\equiv d_{*}$
with $d_{*}'\ind S$ and $d_{*}'$ has color $s_{0}$.

Let $d_{*}'\equiv d_{*}$ be of color $s_{0}$ such that $d_{*}'\ind d_{*}$.
For $i<n$, let $c_{*,i}'\subseteq d_{*}'$ be the tuple corresponding
to $g_{i}^{-1}\left(c_{*}\right)$, so in particular $c_{*,i}'\equiv c_{*}$.
Let $j_{i}<l$ be the color of $c_{*,i}'$. By the choice of $s_{0}$,
for every finite set $S\subseteq M_{*}$, there is some $c_{*}'\equiv c_{*}$
such that $c'_{*}\ind S$ and $c_{*}'$ has color $j_{i}$. 

Since $M$ is homogeneous we can extend $f$ in such a way so that
its domain equals $D_{*}$. By Claim \ref{claim:Using shiftiness}
(the first part), there is some $k_{n,0}<\omega$ and an extension
$f_{0}$ of $f$ that ensures that $\sigma_{*}^{k_{n,0}}\left(d_{*}'\right)=d_{*}$. 

Starting with $f_{0}$, we construct an increasing sequence $\sequence{f_{i}}{i<n}$
as follows. Suppose we have $f_{i}$ whose domain is $D_{*,i}$. Find
some $c''_{*,i+1}\equiv c_{*}$ of color $j_{i+1}$ such that $c_{*,i+1}''\ind D_{*,i}$.
By Claim \ref{claim:Using shiftiness}, we can find $k_{n,i+1}<\omega$
and extend $f_{i}$ to $f_{i+1}$ which ensures that $\sigma_{*}^{k_{n,i+1}}\left(c_{*,i+1}''\right)=c_{*}$. 

Now we have the first part of $\star$: we need to check that $d\left(\sigma_{*}^{k_{n,i}}\left(x_{*}\right),g_{i}\left(\sigma_{*}^{k_{n,0}}\left(x_{*}\right)\right)\right)<\varepsilon$
for all $i<n$. For $i=0$ this is clear since $g_{0}=\id$, so we
may assume that $i>0$. As $\sigma_{*}^{k_{n,i}}\left(c_{*,i}''\right)=c_{*}$,
it follows that $d\left(\sigma_{*}^{k_{n,i}}\left(x_{0}\right),z_{j_{i}}\right)<\varepsilon/2$.
Similarly, as $\sigma_{*}^{k_{n,0}}\left(c_{*,i}'\right)=g_{i}^{-1}\left(c_{*}\right)$,
we have that $d\left(g_{i}\left(\sigma_{*}^{k_{n,0}}\left(x_{0}\right)\right),z_{j_{i}}\right)<\varepsilon/2$.
Together, we are done. 

Now we have to take care of the other half of $\star$. This is done
similarly, using right existence and the second part of Claim \ref{claim:Using shiftiness}. 
\end{proof}
The following proposition explains why we needed to take a conjugate
of $\sigma$. The countable ordered random graph has a CIR by Example
\ref{exa:The-random-ordered}, thus Theorem \ref{thm:contains a subflow}
applies to it. In Section \ref{subsec:Ramsey, top dynamics}, we mentioned
that it is a Ramsey structure. Note that the underlying order is dense
(by Proposition \ref{prop:restriction of a mix of two Faisse}). 
\begin{prop}
\label{prop:need to take conjugate in the ordered random graph}Let
$M=\left(V,<,R\right)$ be the countable ordered random graph. Then
there is no automorphism $\sigma\in G=\Aut\left(M\right)$ which satisfies
({*}) for every continuous action on a compact metric space $X$ on
which $G$ acts and every $x_{*}\in X$. 
\end{prop}

\begin{proof}
First we find $a\neq b$ in $M$ such that $\sigma^{n}\left(a\right)\neq\sigma^{m}\left(b\right)$
for all $m,n\in\Zz$. To do that, take any $a\in M$. Then $\set{\sigma^{n}\left(a\right)}{n\in\Zz}$
is discrete (in the order sense: it is either a $\Zz$-chain or just
$a$). Since $\left(V,<\right)$ is dense, there is some $b\neq\sigma^{n}\left(a\right)$
for all $n\in\Zz$. It follows that $b$ is as required.  Let $X=S_{x}\left(M\right)$
be the space of complete types over $M$ (in one variable $x$) (it
is a compact metric space). Let $p\in X$ be any completion of the
partial type $\set{R\left(x,\sigma^{n}\left(a\right)\right)}{n\in\Zz}\cup\set{\neg R\left(x,\sigma^{m}\left(b\right)\right)}{m\in\Zz}$.
Then if ({*}) holds for $p$, then by Fact \ref{fact:ramsey iff extremeley amenable},
there is some point $p_{0}\in\cl\set{\sigma^{n}\left(p\right)}{n<\omega}$
which is a fixed point of $G$. In other words, $p_{0}$ is an invariant
type over $M$. However  $R\left(x,a\right)\land\neg R\left(x,b\right)\in p_{0}$
(this is true for any type in the closure), so $p_{0}$ cannot be
invariant (because $G$ is transitive). 
\end{proof}
The example of the ordered random graph also explains why we needed
to restrict to compact metric spaces, and could not prove this for
all compact spaces. If Theorem \ref{thm:contains a subflow} had worked
for all compact spaces, it would also work for the universal $G$-ambit
(see Section \ref{subsec:Ramsey, top dynamics}), $\left(X,x_{0}\right)$.
Thus, there would be a conjugate $\sigma_{*}$ of $\sigma$ such that
$\cl\set{\sigma_{*}^{n}\left(x_{0}\right)}{n<\omega}$ contains a
subflow. But then if $\left(Y,y_{0}\right)$ is any other $G$-ambit,
by universality, there is a continuous surjection $\pi:X\to Y$ mapping
$x_{0}$ to $y_{0}$ and commuting with the action of $G$. Thus,
$\pi$ maps $\cl\set{\sigma_{*}^{n}\left(x_{0}\right)}{n<\omega}$
to $\cl\set{\sigma_{*}^{n}\left(y\right)}{n<\omega}$, and the latter
contains a $G$-subflow.  Thus we get that $\sigma_{*}$ satisfies
({*}) for every $G$-ambit, which contradicts Proposition \ref{prop:need to take conjugate in the ordered random graph}. 
\begin{cor}
\label{cor:Dense trees don't have a groovy ind relation}Let $T=\DenseTrees$
be the theory of dense trees in the language $\left\{ <,\wedge\right\} $,
and let $M\models T$ be countable. Then $\Aut\left(M\right)$ has
no shifty automorphism. In particular, $M$ has no CIR.

Furthermore, the same is true for $T_{dt,<_{lex}}$, the theory of
the lexicographically ordered dense tree $N$, see Example \ref{exa:lex trees are ramsey}.
\end{cor}

\begin{proof}
Suppose that $\sigma$ was shifty. Let $\bar{m}=\sequence m{m\in M}$
be an enumeration of $M$ (really the identity function), and let
$\bar{x}=\sequence{x_{m}}{m\in M}$. Let $X=S_{\bar{m}}\left(M\right)$
be the space of $\bar{x}$-complete types $p$ over $M$ such that
$p\restriction\emptyset=\tp\left(\bar{m}/M\right)$. Then $X$ is
a compact metric space. Let $x_{*}=\tp\left(\bar{m}/M\right)$. By
Theorem \ref{thm:contains a subflow}, there is some conjugate $\tau$
of $\sigma$ such that $\cl\set{\tau^{n}\left(x_{*}\right)}{n<\omega}$
contains a subflow $Y^{+}\subseteq X$ and similarly, $\cl\set{\tau^{-n}\left(x_{*}\right)}{n<\omega}$
contains a subflow $Y^{-}$. By Proposition \ref{prop:in trees, every automorphism fixes a branch or a point},
$\tau$ fixes a branch or a point. 

Suppose that $\tau\left(m\right)=m$ for some $m\in M$. Then for
every $p\in Y^{+}$, $p\models x_{m}=m$. However $G=\Aut\left(M\right)$
acts transitively on $M$, so we have a contradiction. 

Now suppose that $\tau$ fixes a branch $B\subseteq M$, but does
not fix any point. Suppose that $\tau\left(m\right)>m$ for some $m\in B$.
Then $\tau^{n}\left(m\right)>m$ for all $n<\omega$, so for any $p\in Y^{+}$,
$p\models x_{m}>m$. There is some $m'\in M$ such that $m'>m$ and
$m'\notin B$. Since $m<\tau^{n}\left(m\right)\in B$ for all $n<\omega$,
it follows that $p\models x_{m}\wedge m'=m$ for all $p\in Y^{+}$.
Let $\tau'\in G$ fix $m$ and map $m'$ to $B$. Then $\tau'\left(p\right)\models\left(x_{m}\wedge\tau'\left(m'\right)\right)=m<x_{m}$.
But $\tau'\left(p\right)\in Y^{+}$, so $\tau'\left(p\right)\models x_{m}\leq\tau'\left(m'\right)\vee\tau'\left(m'\right)\leq x_{m}$,
which is a contradiction. If, on the other hand $\tau\left(m\right)<m$,
then $\tau^{-1}\left(m\right)>m$, so we can apply the same argument
to $Y^{-}$. 

For the furthermore part, note that by Proposition \ref{prop:restriction of a mix of two Faisse},
the reduct of $T_{dt,<_{lex}}$ to the tree language is $\DenseTrees$.
In addition, letting $H=\Aut\left(N\right)$, $H$ acts transitively
on $N$ (by quantifier elimination, as $N$ is ultrahomogeneous).
In addition, if $B\subseteq N$ is a branch, $m\in B$, there is always
some $m'>m$, $m'\notin B$ and for any $n'>m$ in $B$, $m'm\equiv n'm$.
Hence, we can apply Proposition \ref{prop:in trees, every automorphism fixes a branch or a point}
and the same proof will work. 
\end{proof}

\section{\label{sec:Further-questions}Further questions}

The results presented in the previous sections lead to a number of
questions, both related to CIR and more generally on $\omega$-categorical
structures. We state here a few general conjectures and questions.
If they turn out to be false at this level of generality, they could
be weakened by restricting to finitely homogeneous structures or other
subclasses.

The following conjecture, along with Theorem \ref{thm:If no CQ then finitely generated}
(and Example \ref{exa:If G0=00003DG, then no compact quotients}),
would imply that indeed compact quotients are the only obstruction
to having finite topological rank.
\begin{conjecture}
\label{conj:expansion CIR}Any $\omega$-categorical structure has
an $\omega$-categorical expansion which admits a CIR.
\end{conjecture}

Suppose that $M$ is a structure and $\C$ a monster model for $\Th\left(M\right)$.
The group of\emph{ Lascar strong automorphisms of $M$}, denoted by
$\Aut f\left(M\right)$ is the group of automorphisms of $M$ generated
by the set $\set{\sigma\restriction M}{\exists N\prec\C,\left|N\right|=\left|T\right|,\sigma\restriction N=\id}$.
If $\sigma$ is Lascar strong, then $\sigma\restriction\acl^{\eq}\left(\emptyset\right)=\id$
so $\Aut f\left(M\right)$ is contained in $G^{0}$. However, there
are examples (even $\omega$-categorical examples) where $G^{0}$
is strictly bigger than $\Aut f\left(M\right)$, see \cite{ivanov2010countably,pelaez2008lascar}.
The\emph{ Lascar group }of $M$\emph{ }is the quotient $\Aut\left(M\right)/\Aut f\left(M\right)$.
For more on the Lascar group, see \cite{MR2018391}. In the $\omega$-categorical
case, the quotient $\Aut\left(\acl^{\eq}\left(\emptyset\right)\right)$
is also called the \emph{compact Lascar group}. 

If $M$ is an ultrahomogeneous linearly ordered Ramsey structure,
then by  Proposition \ref{prop:Ramsey has everything but over base},
there is some model $N$ such that $N\ind^{ns}M$.  In particular,
$\sigma\left(M\right)\equiv_{N}M$, for every $\sigma\in\Aut\left(M\right)$
which implies that $\sigma$ is Lascar strong. Thus in Ramsey structures,
and in fact for any model $M$ for which there is some such $N$,
the Lascar group is trivial, and there are no compact quotients. For
instance, by Lemma \ref{lem:CIR implies just over 0} this happens
also when $M$ is $\omega$-categorical with a CIR. 

As we said above, we conjecture that if $\Aut\left(M\right)$ has
no compact quotients then it has finite topological rank. However,
as we pointed out, it could be that $G^{0}=G$ but the Lascar group
is nontrivial. Thus, potentially, the Lascar group \textemdash{} as
a quotient of $\Aut\left(M\right)$ \textemdash{} can be an obstruction
to having finite topological rank. During a talk given on this paper
by the second author, Anand Pillay asked if this scenario could happen.
Conjecture \ref{conj:expansion Lascar} (together with Theorem \ref{thm:If no CQ then finitely generated})
implies that it could not. 
\begin{conjecture}
\label{conj:expansion Lascar}Any $\omega$-categorical structure
has an $\omega$-categorical expansion with trivial Lascar group.
\end{conjecture}

By the above, this second conjecture is implied by Conjecture \ref{conj:expansion CIR}.

Note also that by Proposition \ref{prop:G0 acts oligomorphically and has no CQ},
the conjecture is true when we replace the Lascar group by the compact
Lascar group. 

It would be interesting to investigate other consequences of having
a CIR. For instance a CIR might have something to say about normal
subgroups. The analysis in \cite{DroHolMac} of automorphism groups
of trees seems to suggest that there is a link: normal subgroups appear
as groups fixing a set of points roughly corresponding to the set
of $x$ such that $x\ind A$ for some CIR $\ind$ and finite set $A$.
A similar phenomenon happens in DLO, where there are only three normal
subgroups (the group of automorphism fixing a cone to the left, to
the right, and the intersection of these two), see \cite[Theorem 2.3.2]{MR645351}. 

In another direction, recall that an automorphism group $G$ (or more
generally a Polish group) has the \emph{small index property} (sip)
if every subgroup of index less than $2^{\aleph_{0}}$ is open. Many
groups are known to have this property, but there are at least two
different types of techniques used to show it\textemdash the Hrushovski
property (or extension property) and direct combinatorial methods\textemdash which
have yet to be unified. We refer to \cite{Mac_survey} for a survey
on this. As in the case of finite topological rank, large compact
quotients seem to be only known obstruction to having sip, although
the situation is more complicated: Lascar \cite[end of Section 2.2]{Lascar_petitindice}
gives an example of an automorphism group without the sip and with
no compact quotients. In fact the compact quotients are hidden in
the stabilizer of a finite set. It seems that one can avoid this counterexample
by restricting to dense subgroups. This leads us to the following
questions.
\begin{question}
Let $M$ be $\omega$-categorical such that $G=\Aut\left(M\right)$
has no compact quotient. Is it true that any dense subgroup of $G$
of index less than $2^{\aleph_{0}}$ is open (and hence is equal to
$G$)?
\end{question}

Note that the assumption of having no compact quotient is necessary.
Indeed, in the example suggested by Cherlin and Hrushovski (the one
described in Remark \ref{rem:constraint for cdcc compact quotient}),
we have that $G=\Aut\left(M\right)$ has a dense subgroup of index
2, see \cite[Section 2.1]{Lascar_petitindice}. 
\begin{question}
Let $M$ be $\omega$-categorical and $N$ an $\omega$-categorical
expansion of $M$. Set $G=\Aut\left(M\right)$ and $H=\Aut\left(N\right)\leq G$.
Assume that $(G,H)$ has no compact quotients and that $H$ has the
sip. Is it true that any dense subgroup of $G$ of index less than
$2^{\aleph_{0}}$ is open?
\end{question}

\section*{Acknowlegements}

Thanks to Alejandra Garrido for bringing up some questions that lead
to this work and to Dugald Macpherson for helping us get a grasp of
the area through several interesting discussions. We would also like
to thank the organizers of the 2016 Permutation Groups workshop in
Banff, during which those interactions took place. Thanks to Katrin
Tent for comments on a previous draft and for telling us about \cite{Kwiatkowska2016}.

We would also like to thank Daoud Siniora for his comments. 

Finally, we would like to thank the anonymous referee for his comments. 

\bibliographystyle{alpha}
\bibliography{common}

\end{document}